\ifx \processToSlides \undefined
  \ifx\usepackage\RequirePackage    % documentclass already done...  Assume
    \let\usingAmsArtXII\usepackage  % this file is processed by amsart
  \fi
\else
  \documentclass[12pt]{amsart}
  \usepackage{amssymb,amscd}
  \usepackage[letterpaper,margin=2cm,includeheadfoot]{geometry}
  \def \useHugeSize {}
  \def \numberingIsThrough {}
\fi

\ifx \labelsONmargin\undefined

  \ifx\RequirePackage\undefined
    \expandafter\ifx\csname amsart.sty\endcsname\relax
        \documentstyle[12pt,amssymb,amscd]{amsart}
    \fi

    \def\atSign{@@}

    \def\mathbb{\Bbb}
    \def\mathfrak{\frak}
    \def\mathbf{\bold}
    \ifx\boldsymbol\undefined   % I *think* some very old AMS* did not have it
      \def\boldsymbol#1{{\bold #1}}
    \fi

    \def\mathbit{\boldsymbol}

    \makeatletter
    \newenvironment{proof}{%
         \@ifnextchar[{%
                       \expandafter\let\expandafter\end@proof
                         \csname endpf*\endcsname
                         \my@proof
                      }{\let\end@proof\endpf\pf}%
        }{\end@proof}
    \def\my@proof[#1]{\@nameuse{pf*}{#1}}
    \makeatother

    \def\xrightarrow[#1]#2{@>{#2}>{#1}>}
    \def\xleftarrow[#1]#2{@<{#2}<{#1}<}

    \def\providecommand#1{\def#1}
    \def\emph#1{{\em #1}}
    \def\textbf#1{{\bf #1}}

    \def\mathring{\overset{\,\,{}_\circ}}% For slanted letters only, sub too high
  \else
      \ifx\usepackage\RequirePackage
      \else
        \documentclass[12pt]{amsart}
        \usepackage{amssymb,amscd}
    \let\usingAmsArtXII\usepackage
      \fi
      \ifx \mathring\undefined      % Not AmS-LaTeX 2
        \DeclareMathAccent{\mathring}{\mathalpha}{operators}{"17}
      \fi

      \long\def\FAKEendPROOF{\endtrivlist}
      \ifx\endproof\FAKEendPROOF
      \def\endproof{\qed\endtrivlist}
      \fi

      \ifx\mathbit\undefined    % They can collect their senses and define it
        \DeclareMathAlphabet{\mathbit}{OML}{cmm}{b}{it}
      \fi

      \def\atSign{@}

      \ifx \usingAmsArtXII \undefined
      \else             % We loaded this package ourselves
    \advance\oddsidemargin -0.1\textwidth
    \evensidemargin\oddsidemargin
      \fi

      \def\Sb#1\endSb{_{\substack{#1}}}
      \def\Sp#1\endSp{^{\substack{#1}}}

  \fi
\makeatletter
\@ifundefined{DeclareFontShape}
     {\@ifundefined{selectfont}
             {%    Here we are in non-NFSS
                \message{We need AmS-LaTeX, this version of LaTeX does not
                        support it}\@@end
             }{%  Here we are in  NFSS1
                \def\mathcal{\cal}
                
                \def\pcyr{%
                        \def\default@family{UWCyr}%
                        \let\oldSl@\sl
                        \def\sl{\def\default@shape{it}\oldSl@}%
                        \cyracc
                        \language\Russian\family{UWCyr}\selectfont
                }
             }}{% Here we are in  NFSS2
                \DeclareFontEncoding{OT2}{}{\noaccents@}
                \DeclareFontSubstitution{OT2}{cmr}{m}{n}
                \DeclareFontFamily{OT2}{cmr}{\hyphenchar\font45 }
                \DeclareFontShape{OT2}{cmr}{m}{n}{%
                     <5><6><7><8><9><10>gen*wncyr %
                     <10.95><12><14.4><17.28><20.74><24.88> wncyr10 %
                }{}
                \DeclareFontShape{OT2}{cmr}{m}{it}{%
                     <5><6><7><8><9><10> gen * wncyi%
                     <10.95><12><14.4><17.28><20.74><24.88> wncyi10%
                }{}
                \DeclareFontShape{OT2}{cmr}{bx}{n}{%
                     <5><6><7><8><9><10> gen * wncyb%
                     <10.95><12><14.4><17.28><20.74><24.88> wncyb10%
                }{}
                \DeclareFontShape{OT2}{cmr}{m}{sl}{%
                     <-> ssub * cmr/m/it%
                }{}
                \DeclareFontShape{OT2}{cmr}{m}{sc}{%
                     <5><6><7><8><9><10>%
                     <10.95><12><14.4><17.28><20.74><24.88> wncysc10%
                }{}
                \DeclareFontFamily{OT2}{cmss}{\hyphenchar\font45 }
                \DeclareFontShape{OT2}{cmss}{m}{n}{%
                     <8><9><10> gen * wncyss%
                     <10.95><12><14.4><17.28><20.74><24.88> wncyss10%
                }{}
                
                \def\cyrencodingdefault{OT2}
                \def\pcyr{%
                        \cyracc
                        \let\encodingdefault\cyrencodingdefault
                        \language\Russian\fontencoding{OT2}\selectfont
                }
             }

\@ifundefined{theorembodyfont}
     {
        \def\theorembodyfont#1{\relax}
        \makeatletter
          \let\@@th@plain\th@plain
          \def\th@plain{ \@@th@plain \slshape }
        \makeatother
        \let\normalshape\relax
     }{}

\ifx\cprime\undefined
     \def\cprime{$'$}
\fi
\makeatletter
  \def\@sect@my#1#2#3#4#5#6[#7]#8{%
\ifnum #2>\c@secnumdepth
   \let\@svsec\@empty
 \else
   \refstepcounter{#1}%
\edef\@svsec{\ifnum#2<\@m
             \@ifundefined{#1name}{}{\csname #1name\endcsname\ }\fi
\noexpand\rom{\csname the#1\endcsname.}\enspace}\fi
 \@tempskipa #5\relax
 \ifdim \@tempskipa>\z@ % then this is not a run-in section heading
   \begingroup #6\relax
   \@hangfrom{\hskip #3\relax\@svsec}{\interlinepenalty\@M #8\par}%
   \endgroup
   \if@article\else\csname #1mark\endcsname{%
        \ifnum \c@secnumdepth >#2\relax\csname the#1\endcsname. \fi#7}\fi
\ifnum#2>\@m \else
       \let\@tempf\\ \def\\{\protect\\}\addcontentsline{toc}{#1}%
{\ifnum #2>\c@secnumdepth \else
             \protect\numberline{%
               \ifnum#2<\@m
               \@ifundefined{#1name}{}{\csname #1name\endcsname\ }\fi
               \csname the#1\endcsname.}\fi
           #8}\let\\\@tempf
     \fi
 \else
  \def\@svsechd{#6\hskip #3\@svsec
    \@ifnotempty{#8}{\ignorespaces#8\unskip
       \ifnum\spacefactor<1001.\fi}%
        \ifnum#2>\@m \else
          \let\@tempf\\ \def\\{\protect\\}\addcontentsline{toc}{#1}%
            {\ifnum #2>\c@secnumdepth \else
              \protect\numberline{%
                \ifnum#2<\@m
                \@ifundefined{#1name}{}{\csname #1name\endcsname\ }\fi
                \csname the#1\endcsname.}\fi
             #8}\let\\\@tempf\fi}%
 \fi
\@xsect{#5}}
  \ifx\RequirePackage\undefined
  \let\@sect\@sect@my             % Cannot just comment the above
  \fi
  \ifx\RequirePackage\undefined %              Not under LaTeX2e
  \def\th@remark@my{\theorempreskipamount6\p@\@plus6\p@
    \theorempostskipamount\theorempreskipamount
    \def\theorem@headerfont{\it}\normalshape}
    \let\th@remark\th@remark@my
  \else             %              Under LaTeX2e
    \let\o@@remark\th@remark

    \ifx\theorempostskipamount\undefined        % AmS-art
    \else                       % Transformation groups
      \def\th@remark{\o@@remark
    \ifdim\theorempostskipamount < 2pt\relax
      \theorempostskipamount\theorempreskipamount
         \multiply\theorempostskipamount\tw@
         \divide\theorempostskipamount\thr@@
    \fi
      }
    \fi
  \fi
\let\myLabel\@gobble
\def\labelsONmargin{\@mparswitchfalse\def\myLabel##1{\@bsphack\marginpar
                                  {\normalshape\tiny\rm Label ##1}\@esphack}}
\ifx \url\undefined
  \def\url#1{{\tt #1}}%
\fi
\def\cyracc{\def\u##1{%\if \i##1\accent"24 i%
                \if \i##1\char"1A%
                \else \if I##1\char"12%
                \else \accent"24 ##1\fi\fi }%
\def\"##1{\if e##1{\char"1B}%
                \else \if E##1{\char"13}%
                \else \accent"7F ##1\fi\fi }%
\def\9##1{\if##1z\char"19
\else\if##1Z\char"11
\else\if##1E\char"03
\else\if##1e\char"0B
\else\if##1u\char"18
\else\if##1U\char"10
\else\if##1A\char"17
\else\if##1a\char"1F
\else\if##1p\char"7E
\else\if##1P\char"5E
\else\if##1Q\char"5F
\else\if##1q\char"7F
\else\if##1i\char"1A
\else\if##1I\char"12
\else\if##1N\char"7D
\fi
\fi
\fi
\fi
\fi
\fi
\fi
\fi
\fi
\fi
\fi
\fi
\fi
\fi
\fi
}%
\def\cydot{{\kern0pt}}}%
\def\cydot{$\cdot$}
\ifx\Russian\undefined
        \def\Russian{0\relax
    \message{Don't know the hyphenation rules for Russian^^J
                        Please do INITeX with `input  russhyph' in the
                        command line}%
                \gdef\Russian{0\relax}%
        }
\fi
\ifx\RequirePackage\undefined           % for 209
  \def\@putname#1#2#3#4{\def\@@ref{#3}\let\old@bf\bf
        \def\bf##1{\old@bf\if?\noexpand##1?{#4}\else##1\fi}%
    #1{#2}%
        \let\bf\old@bf}
\else                       % for 2e
  \def\@putname#1#2#3#4{\def\@@ref{#3}\let\old@bf\bf    % for 209???
    \let\old@reset@font\reset@font          % for 2e
        \def\bf##1{\old@bf\if?\noexpand##1?{#4}\else##1\fi}%
    \def\reset@font##1##2{\old@reset@font##1\if?\noexpand##2?{#4}\else##2\fi}#1{#2}%
        \let\bf\old@bf\let\reset@font\old@reset@font}
\fi
\let\my@ref=\ref
\def\ref#1{\@putname\my@ref{#1}{#1}{\tiny\rm\@@ref}}
\let\my@pageref=\pageref
\def\pageref#1{\@putname\my@pageref{#1}{#1}{\tiny\rm\@@ref}}
\let\my@cite=\cite
\def\cite#1{\@putname\my@cite{#1}{\@citeb}{\tiny\rm\@@ref}}
\makeatother
\fi
\relax
\theoremstyle{plain} % for references in unnumbered theorems
\theorembodyfont{\sl}

\ifx\useDoubleSpacing\undefined
\else
  \usepackage{setspace}
\fi

\ifx \numberingIsThrough \undefined
\numberwithin{equation}{section}
\fi
\input xy
\xyoption{all}
\theorembodyfont{\rm}
\theoremstyle{definition}

\ifx \numberingIsThrough \undefined
\newtheorem{definition}{Definition}[section]
\else
\newtheorem{definition}{Definition}
\fi

\theoremstyle{remark}

\newtheorem{remark}[definition]{Remark} %\renewcommand{\theremark}{}

\ifx \numberingIsThrough \undefined
\newtheorem{note}{Note}[section] 
\newtheorem{summary}{Summary}[section] 
\else

\fi

\theoremstyle{plain} % for future references
\theorembodyfont{\sl}

\newtheorem{theorem}[definition]{Theorem}
\newtheorem{lemma}[definition]{Lemma}
\newtheorem{corollary}[definition]{Corollary}
\newtheorem{proposition}[definition]{Proposition}

\begin{document}
\bibliographystyle{amsplain}

\ifx\useHugeSize\undefined
\else
\Huge
\fi

\relax

\title[Weight modules]{On weight modules of algebras of twisted differential
operators on the projective space}

\author{ Dimitar Grantcharov and Vera Serganova }

%\date{ \today }

\thanks{D.G is supported by NSA grant H98230-13-1-0245;  V.S. is supported by  NSF grant 0901554}

\address{ Dept. of Mathematics, University of Texas at Arlington,
Arlington, TX 76019 } \email{grandim\atSign{}uta.edu}

\address{ Dept. of Mathematics, University of California at Berkeley,
Berkeley, CA 94720 } \email{serganov\atSign{}math.berkeley.edu}

\maketitle

\begin{abstract}
We classify blocks of categories of weight and generalized weight modules of algebras of twisted differential operators on ${\mathbb P}^n$. Necessary and sufficient conditions for these blocks to be tame and proofs that some of the blocks are Koszul are provided. We also establish equivalences of categories between these blocks and categories of bounded and generalized bounded  weight $\mathfrak{sl} (n+1)$-modules in the cases of nonintegral and singular central character.\\


\noindent : Lie algebra,
indecomposable representations, quiver, weight modules, twisted differential operators 

\end{abstract}

\section{Introduction}

Algebras of twisted differential operators play important role in  modern representation theory. One fundamental  application of these algebras is the equivalence of categories of modules over a complex semsimple Lie algebra and modules of global sections of algebras of twisted differential operators established by Beilinson and Bernstein, \cite{BB}. In this paper we initiate the study of generalized weight modules of such algebras, i.e. modules that are locally finite with respect to a maximal commutative subalgebra.

The first step in our study is to look at  the category of generalized weight modules over the Weyl algebra $\mathcal D(n+1)$.
We classify the blocks of this category and prove that these blocks
are equivalent to the category of locally nilpotent $G$-graded 
representations of the polynomial algebra $\mathbb C[z_0,\dots,z_n]$,
where $G$ is the direct product of several copies of $\mathbb Z_2$.
Weight modules of the Weyl algebra were studied by Bavula, Bekkert, Benkart, Futorny, among others, \cite{Ba}, \cite{BB}, \cite{BBF}, \cite{GS}.

The original motivation of this study comes from the problem of classifying
blocks of generalized weight modules with bounded set of weight multiplicities over the Lie algebra $\mathfrak{sl}(n+1)$. Such modules are called generalized bounded modules. The natural action of $SL(n+1)$ on
the affine space $\mathbb A^{n+1}$ induces a homomorphism  
$\psi:U(\mathfrak{sl}(n+1)) \to \mathcal D(n+1)$. The image of this
homomorphism lies in the subalgebra 
$\mathcal D^E\subset\mathcal D(n+1)$ of operators
commuting with the Euler vector field $E$. In many cases generalized
bounded weight $\mathfrak{sl}(n+1)$-modules can be obtained from
$\mathcal D^E$-modules via $\psi$. If $Z$ is the center of 
$U(\mathfrak{sl}(n+1))$, then $\psi(Z)$ lies in the subalgebra of $ \mathcal D(n+1)$
generated by $E$. Hence, all $\mathfrak{sl}(n+1)$-modules  with a fixed
central character come from $\mathcal D^E/(E-a)$-modules for a fixed
$a\in\mathbb C$. Note that   $\mathcal D^a=\mathcal D^E/(E-a)$ 
is nothing else but
the algebra of global section of twisted differential operators on the
projective space $\mathbb P^n$. 

It is not hard to see that any indecomposable  generalized weight
module of finite length admits a generalized central
character. Therefore, it is important to study generalized weight
modules over the ring 
$\displaystyle \widetilde{\mathcal D}^a =\lim_{\longleftarrow}\mathcal D^E/(E-a)^n$, 
which itself can be described
as the ring of formal deformations of  $\mathcal D^a$.
We classify the blocks of generalized weight $\widetilde{\mathcal D}^a$-modules 
in terms of quivers with relations. 
Then we prove that the category of generalized weight 
$\mathfrak{sl}(n+1)$-modules 
with bounded weight multiplicities that admit non-integral or singular
integral generalized central character is equivalent to the category
of generalized weight $\widetilde{\mathcal D}^a$-modules for a suitable
choice of $a$. The proof uses variations of the twisted localization construction
originally introduced in \cite{M} and the classification of generalized
weight cuspidal blocks, see \cite{MS}. We also obtain all  classification 
results and equivalences of categories mentioned above in the case of weight modules - those for which the action of the corresponding commutative subalgebra is semisimple.
 
Let us mention that the above result for a non-integral or singular
integral central character is essentially a consequence of the fact
that in this case
all simple bounded weight modules with the same central character
are annihilated by the same primitive ideal. If the central
character is regular integral, then  we have $n+1$ primitive ideals representing annihilators of the modules in the corresponding block. 
This case was solved in \cite{GS2} and in \cite{MS} for cuspidal 
weight modules and for generalized cuspidal weight modules,
respectively. The case of  general bounded modules with regular integral central character remains open in general.

The paper is organized as follows. In Section 3 we collect some important definitions and facts for generalized weight modules of $\mathfrak{sl} (n+1)$ and $\widetilde{\mathcal D}^a$. The quivers and the corresponding algebras that appear in the classification results are introduced in Section 4. The study of generalized weight modules of ${\mathcal D} (n+1)$ is presented in Section 5, while the next section is devoted to the study of these modules over $\displaystyle \widetilde{\mathcal D}^a$. The twisted localization construction is presented in Section 7 and the main results on the equivalence of categories are included in Section 8. Some useful commutative diagrams of functors are collected in the Appendix.

\medskip
\noindent{\it Acknowledgements.}  The authors would like to thank to V. Bekkert for the useful information provided for certain quivers. We also thank L. Hille and A. Polischuk for the fruitful discussions.

\section{Index of notations}
Below we list some notations that are frequently used in the paper under the section number they are  introduced.

\ref{subsec-weight} supp, $M^{\lambda}$, $M^{(\lambda)}$.

\ref{subsec-wtd} $\mathcal D (n+1)$, $\mathcal D$,  $({\mathcal D}, {\mathcal H})$-mod, $^{\rm b}({\mathcal D}, {\mathcal H})$-mod,  $({\mathcal D}, {\mathcal H})_{\nu}$-mod,  $^{\rm b}({\mathcal D}, {\mathcal H})_{\nu}$-mod.

\ref{subsec-wtde} $|\nu|$, $E$, 
$(\mathcal D^E, \mathcal H)^a$-mod, $(\mathcal D^E, \mathcal H)^a_{\nu}$-mod,  $_{\rm s}(\mathcal D^E, \mathcal H)^a$-mod,  $_{\rm s}^{\rm b}(\mathcal D^E, \mathcal H)^a$-mod, $_{\rm s}(\mathcal D^E, \mathcal H)^a_{\nu}$-mod,  $_{\rm s}^{\rm b}(\mathcal D^E, \mathcal H)^a_{\nu}$-mod.

\ref{subsec-wtsl} $\gamma$, $\mathcal{B}$,  $\mathcal{GB}$, $\overline{\mathcal{B}}$, $\overline{\mathcal{GB}}$,   $\mathcal{B}_{\nu}$, $\mathcal{GB}_{\nu}$, $\mathcal{B}^{\lambda}$, $\mathcal{GB}^{\lambda}$, $\mathcal{B}^{\lambda}_{\nu}$, $\mathcal{GB}^{\lambda}_{\nu}$, $\overline{\mathcal{B}}^{\lambda}_{\nu}$, $\overline{\mathcal{GB}}^{\lambda}_{\nu}$.

\ref{subsec-d-weight} $S_{\mathcal H'}$, ${\mathcal B}_{\nu}$, ${\mathcal B} (k)$.

\ref{subsec-def-quivers} $C(k), A(k), A'(k), A''(k), B(k), B'(k), B''(k)$.

\ref{subsec-dn} ${\mathcal F}_{\nu} , {\mathcal F}_{\nu}^{\log} $, $\sigma_J$, $\mathcal I_\nu(J)$, ${\mathcal F}_{\nu}^{\log} (J)$, ${\rm Int} (\nu)$, ${\mathcal P} (\nu)$, ${\mathcal R}_\nu$, ${\mathcal A}_\nu$, $G_k$.

\ref{subsec-de} $\Gamma_a$, $\Phi$, $\pi^a$.

\ref{subsec-deha} ${\mathcal P}_a$, $\Phi'$, $\mathcal R^a_0$.

\ref{subsec-loc-gen} $D_F$, $\Theta_F^{\bf x}$, $\Theta_F^{\bf x}$, $D_F^{\bf x}$.

\ref{subsec-twloc-dh} $D_i^{+}, D_i^{-}$, $D_i^{x,+}, D_i^{x,-}$, $D_{i,j}$, $D_{i,j}^x$.

\ref{subsec-psi} $\psi, \Psi$.

\ref{subsec-gen-case} $S_{\mathfrak{h}}$.

\section{Background}

In this paper the ground field is $\mathbb C$.
All vector spaces, algebras, and tensor
products are assumed to be over $\mathbb C$ unless otherwise
stated.

\subsection {Categories of weight modules of associative algebras} \label{subsec-weight}  Let $\mathcal U$ be a
finitely generated associative unital algebra and $\mathcal H\subset\mathcal U$
be a commutative subalgebra. We assume in addition that 
$\mathcal H=\mathbb C[h_0,\dots,h_n]=S(\mathfrak h)$ 
is a polynomial algebra and
${\rm ad}\, (h) : {\mathcal U} \to {\mathcal U}$ is semisimple for all $h\in\mathfrak h$. In other words, we
have a decomposition
$$\mathcal U=\bigoplus_{\mu\in {{\mathfrak h}^*}}\mathcal U^\mu,$$
where
$$\mathcal U^\mu=\{x\in\mathcal U | [h,x]=\mu(h)x, \forall h\in\mathfrak h\}.$$
Let $Q = {\mathbb Z}\Delta_{\mathcal U}$ be the ${\mathbb Z}$-lattice in  ${\mathfrak h}^*$
generated by $\Delta_{\mathcal U}= \{ \mu \in {\mathfrak h}^* \; | \; U^{\mu} \neq 0\}$. We note that $Q$ is of finite rank since  $\mathcal U$ is finitely
generated. We also obviously have
$U^\mu U^\nu\subset U^{\mu+\nu}$. 

We call a ${\mathcal U}$-module $M$ {\it a generalized weight $({\mathcal U}, {\mathcal H})$-module} if  $M = \bigoplus_{\lambda \in {\mathfrak h}^*} M^{(\lambda)}$, where  
$$
M^{(\lambda)} = \{m\in M | (h_i- \lambda (h_i)\mbox{Id})^N m=0\,\text{for some}\, N>0\, \text{and all}\, i=0,\dots,n\}.
$$
We call $M^{(\lambda)}$ the generalized weight space of $M$ and $\dim
M^{(\lambda)}$ the weight multiplicity of the weight $\lambda$. 
Note that 
\begin{equation}\label{rootweight}
\mathcal U^\mu M^{(\lambda)}\subset M^{(\mu+\lambda)}.
\end{equation}
A generalized weight module $M$ is called a {\it weight module} if $M^{(\lambda)}  = M^{\lambda},$ where 
$$
M^{\lambda} = \{m\in M | (h_i- \lambda(h_i) \mbox{Id}) m=0\,\text{ for all }\, i=0,\dots,n\}.
$$
By $(\mathcal U,\mathcal H) {\rm -mod}$ and  $^{\rm w}(\mathcal U,\mathcal H) {\rm -mod}$ we denote the category of generalized weight modules and weight modules, respectively. Furthermore, by $^{\rm f}(\mathcal U,\mathcal H)$-mod and $^{\rm b}(\mathcal U,\mathcal H)$-mod we denote the subcategories of $(\mathcal U,\mathcal H)$-mod  consisting of modules with finite weight multiplicities and bounded set of weight multiplicities, respectively.  By  $^{\rm wf}(\mathcal U,\mathcal H)$-mod and $^{\rm wb}(\mathcal U,\mathcal H)$-mod we denote the  subcategories of  $^{\rm w}(\mathcal U,\mathcal H) {\rm -mod}$ that are in $^{\rm f}(\mathcal U,\mathcal H) {\rm -mod}$ and  $^{\rm b}(\mathcal U,\mathcal H) {\rm -mod}$, respectively.

For any module $M$ in $(\mathcal U,\mathcal H)$-mod we set 
$$
\mbox{supp} M := \{ \lambda \in {\mathfrak h}^* \; | \; M^{(\lambda)} \neq 0\}
$$
to be the {\it support} of $M$.  It
is clear from (\ref{rootweight}) 
that $\operatorname{Ext}_{\mathcal A}^1(M,N)=0$ if
$(\operatorname{supp}M-\operatorname{supp}N)\cap Q=\emptyset$, where ${\mathcal A}$ is any of the categories of generalized weight modules or weight modules defined above. Then we have

$$(\mathcal U,\mathcal H){ \rm -mod}=\bigoplus_{\overline{\mu} \in {\mathfrak h}^*/Q}(\mathcal U,\mathcal H)_{\overline{\mu}}{\rm -mod},$$ 
where  $(\mathcal U,\mathcal H)_{\overline{\mu}}{\rm -mod}$ denotes
the subcategory of  $(\mathcal D,\mathcal H)$-mod consisting of
modules $M$ with $\operatorname{supp}M\subset \overline{\mu} = \mu+ Q$. We similarly define $^{\rm w}(\mathcal U,\mathcal H)_{\overline{\mu}}{\rm -mod}$, $^{\rm f}(\mathcal U,\mathcal H)_{\overline{\mu}}{\rm -mod}$, $^{\rm b}(\mathcal U,\mathcal H)_{\overline{\mu}}{\rm -mod}$, $^{\rm wf}(\mathcal U,\mathcal H)_{\overline{\mu}}{\rm -mod}$, and $^{\rm wb}(\mathcal U,\mathcal H)_{\overline{\mu}}{\rm -mod}$, and obtain the corresponding support composition where the direct summands are parametrized by elements of ${\mathfrak h}^*/Q$. With a slight abuse of notation, for $\mu \in {\mathfrak h}^{*}$ we set $(\mathcal U,\mathcal H)_{\mu}{\rm -mod}  = (\mathcal U,\mathcal H)_{\overline{\mu}}{\rm -mod}$, etc.

\subsection{Weight ${\mathcal D}$-modules}  \label{subsec-wtd}
Let $\mathcal D(n+1)$ be the Weyl algebra, i.e. the algebra of
differential operators of the ring $\mathbb C[t_0,\dots,t_n]$ and consider $\mathcal U=\mathcal D(n+1)$. 
 When $n\geq 0$ is fixed, we
use the notation $\mathcal D$ for  $\mathcal D(n+1)$. Let
$\mathcal H:=\mathbb C[t_0\partial_0,\dots,t_n\partial_n]$. 
Then $\mathcal H$ is a maximal commutative subalgebra in $\mathcal D$.
Note that the adjoint action of the abelian Lie subalgebra ${\rm Span}(t_0\partial_0,\dots,t_n\partial_n)$ on $\mathcal D$ is
semisimple. We identify ${\mathbb C}^{n+1}$ with the dual space of $\mbox{Span} \{ t_0\partial_0,...,t_n \partial_n \}$, and fix $ \{ \varepsilon_{0},\dots ,\varepsilon_{n} \}$ to be the standard basis of this space, i.e.  $\varepsilon_i (t_j \partial_j) = \delta_{ij}$. Then $Q = \bigoplus_{i=0}^n {\mathbb Z} \varepsilon_i$ is identified with ${\mathbb Z}^{n+1}$, and
\begin{equation*}
\mathcal D=\bigoplus_{\mu\in \mathbb Z^{n+1}}\mathcal D^\mu.
\end{equation*}
Here $\mathcal D^0=\mathcal H$ and each $\mathcal D^\mu$ is a free
left $\mathcal H$-module of rank $1$ with generator 
$\prod_{\mu_i\geq 0}t_i^{\mu_i}\prod_{\mu_j<0}\partial_j^{-\mu_j}$. 

Therefore $^{\rm b}({\mathcal D}, {\mathcal H}){\rm -mod} = {}^{\rm
f}({\mathcal D}, {\mathcal H}){\rm -mod}$ and $^{\rm wb}({\mathcal D},
{\mathcal H}){\rm -mod} ={}^{\rm wf}({\mathcal D}, {\mathcal H}){\rm -mod}$. 
The latter category was studied in \cite{BBF} and \cite{GS} and the former in \cite{BaBe}.

The support of every $({\mathcal D}, {\mathcal H})$-module will be considered as a subset of ${\mathbb C}^{n+1}$ and we have a natural decomposition
$$(\mathcal D,\mathcal H){ \rm -mod}=\bigoplus_{\overline{\nu} \in {{\mathbb C}^{n+1}}/ {\mathbb Z}^{n+1}}(\mathcal D,\mathcal H)_{\overline{\nu}}{\rm -mod},$$ 

As before for $\nu \in {\mathbb C}^{n+1}$ we write $(\mathcal D,\mathcal H)_{\nu}{\rm -mod} = (\mathcal D,\mathcal H)_{\overline{\nu}}{\rm -mod}$. The same applies for the subcategories $^{\rm w}(\mathcal D,\mathcal H)_{\overline{\nu}}{\rm -mod}$, $^{\rm b}(\mathcal D,\mathcal H)_{\overline{\nu}}{\rm -mod} = {}^{\rm f}(\mathcal D,\mathcal H)_{\overline{\nu}}{\rm -mod}$, and $^{\rm wb}(\mathcal D,\mathcal H)_{\overline{\nu}}{\rm -mod} = {}^{\rm wf}(\mathcal D,\mathcal H)_{\overline{\nu}}{\rm -mod}$.

\subsection{Weight ${\mathcal D}^E$-modules} \label{subsec-wtde}

In this subsection we assume $n\geq 1$.
Let $E=\sum_{i=0}^n t_i\partial_i$ be the Euler vector field.
Denote by $\mathcal D^E$ the centralizer of $E$ in $\mathcal D$. 
Note that $\mathcal D$ has a $\mathbb Z$-grading
$\mathcal D=\bigoplus_{m \in {\mathbb Z}} \mathcal D^m$, where 
$\mathcal D^m=\{d\in \mathcal D | [E,d]=md\}$. It is not hard to see
that the center of $\mathcal D^E$ is generated by $E$.  The quotient
algebra $\mathcal D^E/(E-a)$ is the
algebra of global sections of twisted differential operators on 
$\mathbb P^n$.

Let
$a \in \mathbb C$, let $(\mathcal D^E, \mathcal H)^{a}$-mod  be the 
category of generalized weight
$\mathcal D^E$-modules with locally nilpotent action of $E - a$ and
$^{\rm b}(\mathcal D^E, \mathcal H)^{a}$-mod be the subcategory of
$^{\rm b}(\mathcal D^E, \mathcal H)^{a}$-mod  consisting of modules with
finite weight multiplicities.
We have again a decomposition
$$ (\mathcal D^E, \mathcal H)^{a}{\rm -mod}=\bigoplus_{|\nu|=a}
(\mathcal D^E, \mathcal H)^{a}_{\nu}{\rm -mod},$$
where $(\mathcal D^E, \mathcal H)^{a}_{\nu}{\rm -mod}$ is the
subcategory of modules with support in 
$\nu+\sum_{i=0}^{n-1}\mathbb Z(\varepsilon_i-\varepsilon_{i+1})$ and $|\nu|:=\sum_{i=0}^n\nu_i$.

Let $\mathcal H'$ be the
subalgebra of $\mathcal D$ generated by $t_i\partial_i-t_j\partial_j$.
We denote by $_{\rm s}(\mathcal D^E, \mathcal H)$-mod  (respectively, $^{\rm b}_{\rm s}(\mathcal D^E, \mathcal H)$-mod) the the subcategory  
of $(\mathcal D^E, \mathcal H)$-mod  (resp., the subcategory  
of $^{\rm b}(\mathcal D^E, \mathcal H)$-mod) consisting of all modules
semisimple over $\mathcal H'$. Similarly we define the categories  $_{\rm s}(\mathcal D^E, \mathcal H)^a$-mod,  $_{\rm s}^{\rm b}(\mathcal D^E, \mathcal H)^a$-mod, $_{\rm s}(\mathcal D^E, \mathcal H)^a_{\nu}$-mod, and $_{\rm s}^{\rm b}(\mathcal D^E, \mathcal H)^a_{\nu}$-mod.

\subsection{Weight $\mathfrak{sl}(n+1)$-modules} \label{subsec-wtsl}
Let ${\mathfrak g}=\mathfrak{sl}\left(n+1\right) $ and $U= U(\mathfrak{g})$ be its universal enveloping algebra.  
We fix a Cartan subalgebra $ {\mathfrak h}$ of ${\mathfrak
g} $ and denote by $(\, , )$ the Killing form on
$\mathfrak g$. We apply the setting of \S \ref{subsec-weight} with
${\mathcal U} = U$ and $\mathcal H =
S({\mathfrak h})$. We will use the following notation: $\mathcal{GB} =
{}^{\rm b}({\mathcal U}, {\mathcal H}){\rm - mod}$, ${\mathcal B} =
{}^{\rm wb}({\mathcal U}, {\mathcal H}){\rm - mod}$,
$\mathcal{GB}_{\mu} =\mathcal{GB}_{\overline{\mu}} = {}^{\rm b}({\mathcal U}, {\mathcal
H})_{\overline{\mu}}{\rm - mod}$, and $\mathcal{B}_{\mu} = \mathcal{B}_{\overline{\mu}} =
{}^{\rm wb}({\mathcal U}, {\mathcal H})_{\overline{\mu}}{\rm - mod}$. 

A generalized weight module $M$ with finite weight multiplicities will
be called a {\it generalized cuspidal module} if the elements of the root space 
${\mathfrak g}^{\alpha}$ act injectively (and hence bijectively) on
$M$ for all roots $\alpha$ of $\mathfrak g$. 
If $M$ is a weight cuspidal module we will call it simply {\it cuspidal module}.
By $\mathcal{GC}$ and ${\mathcal C}$ we will denote the categories of generalized cuspidal and cuspidal modules, respectively, and the corresponding subcategories defined by the supports will be denoted by $\mathcal{GC}_{\overline{\mu}}$ and ${\mathcal C}_{\overline{\mu}}$. One should note that the simple objects of ${\mathcal B}$ and $\mathcal{GB}$ (as well as those of ${\mathcal C}$ and $\mathcal{GC}$) coincide.

The induced form on ${\mathfrak h}^*$ will be
denoted by $(\, ,)$ as well.  We have that $ {\mathfrak h}^{*} $ is identified with the
subspace of $ {\mathbb C}^{n+1} $ spanned by the simple roots 
$\varepsilon_{0}-\varepsilon_{1},\dots
,\varepsilon_{n-1}-\varepsilon_{n} $.  By $ \gamma $ we
denote the projection $ {\mathbb C}^{n+1} \to
{\mathfrak h}^{*} $ with one-dimensional kernel 
$\mathbb C(\varepsilon_0+\dots+\varepsilon_n)$. 
In this case $ Q \subset {\mathfrak h}^{*}$ is the root lattice. 
By $W$ we denote the Weyl group of ${\mathfrak g}$. Denote by $Z:=Z(U)$ the center of $U$
and let $Z':=\mbox{Hom} (Z, \mathbb C)$ be the set of all central
characters (here $\mbox{Hom}$ stands for 
homomorphisms of unital $\mathbb C$-algebras). By $\chi_{\lambda}\in Z'$ we
denote the central character of the irreducible highest weight
module with highest weight $\lambda$. Recall that
$\chi_{\lambda}=\chi_{\mu}$ iff $\lambda+\rho=w(\mu+\rho)$ for
some element $w$ of the Weyl group $W$, where, as usual, $\rho$
denotes the half-sum of positive roots. We say that
$\chi=\chi_{\lambda}$ is {\it regular} if the stabilizer of
$\lambda+\rho$ in $W$ is trivial (otherwise $\chi$ is called {\it
singular}), and that $\chi=\chi_{\lambda}$ is {\it integral} if
$\lambda$ is in the weight lattice (i.e. in the lattice spanned by the 
 fundamental weights $\gamma(\varepsilon_0+...+\varepsilon_{i-1})$, $i=1,...,n$). We say that two weights $\lambda$ and $\nu\in
\lambda+\Lambda$ are in the same Weyl chamber if for any positive
root $\alpha$ such that $(\lambda,\alpha)\in \mathbb Z$,
$(\lambda,\alpha)\in \mathbb Z_{\geq 0}$ if and only if
$(\mu,\alpha)\in \mathbb Z_{\geq
  0}$. Finally, recall that $\lambda$ is {\it dominant integral} if
$(\lambda,\alpha)\in \mathbb Z_{\geq 0}$ for all positive roots $\alpha$.

One should note that   every generalized bounded
module has finite Jordan--H\"older series (see Lemma 3.3 in \cite{M}). Since the center $Z$ of $U$
preserves weight spaces, it acts locally finitely on the generalized bounded
modules. For every central character $ \chi\in Z'$ let $\mathcal{GB}^{\chi} $ (respectively, $\mathcal{B}^{\chi}, \mathcal{GC}^{\chi}, \mathcal{C}^{\chi}$) denote the category of all generalized bounded modules (respectively, bounded, generalized cuspidal, cuspidal) modules $ M $ with
generalized central character $\chi$, i.e. such that for some $n\left(M\right),
\left(z-\chi\left(z\right)\right)^{n\left(M\right)}=0 $ on $M$ for
all $ z\in Z$. It is clear that every generalized bounded module $ M $ is a
direct sum of finitely many $ M_{i}\in \mathcal{GB}^{\chi_{i}} $. Thus, one
can write

\begin{equation}
\mathcal{GB}=\bigoplus_{\chi \in Z' \atop \bar{\mu} \in {\mathfrak
h}^*/Q}\mathcal{GB}^{\chi}_{\bar{\mu}},\;  \mathcal{GC}=\bigoplus_{\chi \in Z' \atop \bar{\mu} \in {\mathfrak h}^*/Q}\mathcal{GC}^{\chi}_{\bar{\mu}}, \; \mathcal{B}=\bigoplus_{\chi \in Z' \atop  \bar{\mu} \in {\mathfrak h}^*/Q}\mathcal{B}^{\chi}_{\bar{\mu}}, \; \mathcal{C}=\bigoplus_{\chi \in Z' \atop \bar{\mu} \in {\mathfrak h}^*/Q}\mathcal{C}^{\chi}_{\bar{\mu}},
\notag\end{equation} where
$\mathcal{GB}^{\chi}_{\bar{\mu}} =\mathcal{GB}^{\chi}  \cap \mathcal{GB}_{\bar{\mu}}$, etc.  Note that many of the direct summands above are trivial. 

By $ \chi_{\lambda} $ we denote the central character of the simple
highest weight $\mathfrak g$-module with highest weight $ \lambda $. For
simplicity we put $\mathcal{GB}^{\lambda}:=\mathcal{
GB}^{\chi_{\lambda}} $, $ \mathcal{
GB}^{\lambda}_{\bar{\mu}}:= \mathcal{GB}^{\chi_{\lambda}}_{\bar{\mu}}$, etc.

Let $ \overline{\mathcal{B}}$ (respectively, $ \overline{\mathcal{GB}}$) be the full subcategory of all  weight modules (respectively, generalized weight modules)
consisting of $\mathfrak g$-modules 
$ M $ whose  finitely generated submodules
 belong to $ {\mathcal B}$ (respectively, ${\mathcal GB}$). It is not hard to see
that every such $M$ is a direct limit $\displaystyle\lim_{\longrightarrow}M_i$
for some directed system $\{M_i | i\in I\}$ such that each $M_i\in\mathcal{
GB}$ (respectively, $M_i \in {\mathcal B}$).
It implies that the action of the center $Z$ of the universal
enveloping algebra $U$ on $M$ is locally finite and we have
decompositions
\begin{equation}
\overline{\mathcal B}=\bigoplus_{\chi \in Z',\atop \bar{\mu} \in {\mathfrak
h}^*/Q}\overline{\mathcal B}^{\chi}_{\bar{\mu}}, \; \overline{\mathcal{GB}}=\bigoplus_{\chi \in Z' \atop \bar{\mu} \in {\mathfrak
h}^*/Q}\overline{\mathcal{GB}}^{\chi}_{\overline{\mu}}. \notag\end{equation}

 In a similar way we define $ \overline{\mathcal{C}}$ and $ \overline{\mathcal{GC}}$ and obtain their block decompositions. Finally, we set $ \overline{\mathcal{GB}}^{\lambda}_{\bar{\mu}}:=\overline{\mathcal{GB}}^{\chi_{\lambda}}_{\bar{\mu}}$,  $ \overline{\mathcal
B}^{\lambda}_{\bar{\mu}}:=\overline{\mathcal
B}^{\chi_{\lambda}}_{\bar{\mu}}$, etc. For a detailed discussion of $\overline{\mathcal
C}$ we refer the reader to \S 5 of \cite{GS2}.

\section{Some quivers and algebras related to hypercubes} 

In this section we define certain pointed algebras which describe
blocks of categories which we study in the paper.

\subsection{Definitions} \label{subsec-def-quivers}
Let $k\geq 1$ and $C(k)$ be the quiver of the skeleton of the
$k$-dimensional cube: the vertices of the quiver are the same as
those of the cube. Arrows of $C(k)$
correspond to the edges of the cube, each edge  gives rise
to two arrows  in $C(k)$ with opposite orientation. 

We define an
equivalence relation on the paths of $C(k)$ as follows. Two paths with
the same beginning and end and of the same length are equivalent. Let
$B(k)$ be the quotient of $\mathbb C[C(k)]$ by the relations
identifying all equivalent paths.

For instance, the quiver for $B(2)$ is
$$
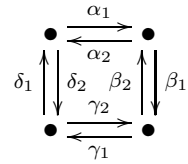
\xymatrix{\bullet  \ar@<0.5ex>[d]^{\delta_2} \ar@<0.5ex>[r]^{\alpha_1}& \bullet \ar@<0.5ex>[d]^{\beta_1} 
\ar@<0.5ex>[l]^{\alpha_2} \\
\bullet  \ar@<0.5ex>[u]^{\delta_1} \ar@<0.5ex>[r]^{\gamma_2}& \bullet
\ar@<0.5ex>[l]^{\gamma_1} \ar@<0.5ex>[u]^{\beta_2}
}
$$
with relations
$\beta_1 \alpha_1 =  \gamma_2 \delta_2, \alpha_2 \alpha_1= \delta_1 \delta_2,...$ etc.

We define the algebras $B'(k)$ for  $k\geq 2$ (respectively, 
$B''(k)$ for $k\geq 3$)
as the path algebras of $C(k)$ with one vertex
(respectively, two opposite vertices) removed subject to the same relations
described above. 
By $B''(2)$ we denote the path algebra of the quiver 
 $$
\xymatrix{\bullet  \ar@<0.5ex>[r]^{}& \bullet
\ar@<0.5ex>[l]^{} }
$$
Note that $B''(2)=B(1)$ by definition.

Next we 
color the arrows of  $C(k)$ by $k$ colors, in such a way, that two arrows have the same
color if and only if the corresponding edges of the cube are parallel. As a result, to every path  in $C(k)$ we associate a sequence of  colors. 
We
call two paths $p$ and $q$ equivalent if 
they have the same beginning and end and
the sequence of colors  of $p$ is obtained from that of $q$ by some permutation.
Let $A(k)$ be the quotient of the path algebra $\mathbb C[C(k)]$ 
by the relations
identifying all equivalent colored paths.

For $k\geq 2$ let $C'(k)$ be obtained from $C(k)$ by removing one vertex from
$C(k)$ and attaching a loop to each vertex adjacent to the removed
one. In this way each loop at a vertex $v$ is a 
replacement of the path of length $2$
from $v$ to the ghost (removed) vertex and back. We again
associate to each path a sequence of colors with the convention that to the
loop at $v$ we associate  the sequence $(c,c)$, where $c$ is the color of
the edge joining $v$ with the ghost vertex. Introduce the equivalence
relation on the paths of $C'(k)$ in the same manner as above and
define $A'(k)$ be the quotient of the path algebra $\mathbb C[C'(k)]$ by the relations
identifying all equivalent paths.

For example $A'(2)$ is the quotient of the path algebra of the quiver
$$
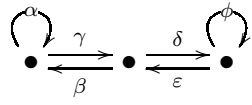
\xymatrix{\bullet  \ar@(ul,ur)|{\alpha} \ar@<0.5ex>[r]^{\gamma}& \bullet
\ar@<0.5ex>[l]^{\beta} \ar@<0.5ex>[r]^{\delta}& \bullet
\ar@<0.5ex>[l]^{\varepsilon} \ar@(ul,ur)[]|{\phi}}
$$
by the  relations 
$$
\alpha \beta \gamma = \beta \gamma \alpha, \, 
\gamma \beta \varepsilon \delta = \varepsilon \delta \gamma \beta, \, 
\delta\varepsilon \phi = \phi \delta\varepsilon.
$$

If $k\geq 3$ we consider also the quiver $C''(k)$ obtained from
$C(k)$ by removing two opposite vertices and attaching loops
to the vertices adjacent to the removed ones. 
We define the coloring of paths and construct the algebra  $A''(k)$ like in the previous case.

For  $k=2$ define $C''(2)$ to be the quiver 
$$
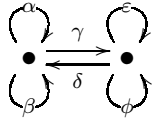
\xymatrix{\bullet \ar@(dl,dr)|{\beta}  \ar@(ul,ur)|{\alpha} \ar@<0.5ex>[r]^{\gamma}& \bullet
\ar@<0.5ex>[l]^{\delta} \ar@(dl,dr)|{\phi} \ar@(ul,ur)[]|{\varepsilon}}
$$
By $A''(2)$ we denote the quotient of  path algebra of $C''(2)$ by the relations
$\alpha \beta = \beta \alpha = \delta \gamma$, $\varepsilon \phi = \phi \varepsilon = \gamma \delta$. 

\subsection{Tameness}
\begin{proposition}\label{wildweight1}  

(a) $B(k)$ is tame if and
only if $k=1,2$.

(b)  $B'(k)$ is tame if and
only if $k=2$.

(c) $B''(k)$ is tame if and
only if $k=2$.

\end{proposition}
\begin{proof} Observe that $B(1)$ is the path algebra of the quiver
$A_1^{(1)}$, which is tame  (see \cite{DF}, \cite{N}). 
To prove the tameness of $B(2)$ we show that it is a skewed-gentle algebra
which is tame by \cite{GP} (see also \cite{BMM}). To prove that it is
skewed-gentle, we  consider the quiver of type $A_1^{(1)}$ with the
relations $\alpha\beta=\beta\alpha=0$, where $\alpha$ and $\beta$ are
two arrows of the quiver, and set both vertices to be special. Then
the corresponding skewed-gentle algebra coincides with $B(2)$.

 The wildness of 
$B'(k)$ and $B''(\ell)$ for $k \geq 3$ and $\ell\geq 4$ 
follows from the fact that both algebras have a quotient algebra 
which is isomorphic to the path algebra of the quiver $$
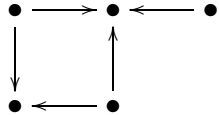
\xymatrix{\bullet  \ar[r] \ar[d]& \bullet &  \bullet \ar[l] \\
\bullet   & \bullet  \ar[l] \ar[u] & }
$$
without relations. The tameness of $B'(2)$ is a standard fact - see for example \S
VIII.7.9 in \cite{Er}. 

It remains to show that $B''(3)$ is wild. 
The quiver for $B''(3)$ is
$$
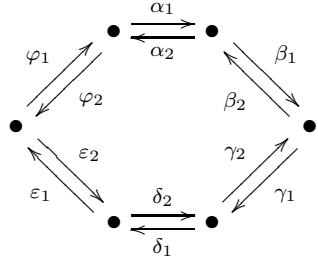
\xymatrix{&\bullet  \ar@<0.5ex>[dl]^{\varphi_2} \ar@<0.5ex>[r]^{\alpha_1}& \bullet \ar@<0.5ex>[dr]^{\beta_1} 
\ar@<0.5ex>[l]^{\alpha_2} & \\
\bullet  \ar@<0.5ex>[ur]^{\varphi_1} \ar@<0.5ex>[dr]^{\varepsilon_2}& && \bullet
\ar@<0.5ex>[dl]^{\gamma_1} \ar@<0.5ex>[ul]^{\beta_2}\\
&\bullet  \ar@<0.5ex>[ul]^{\varepsilon_1} \ar@<0.5ex>[r]^{\delta_2}& \bullet \ar@<0.5ex>[ur]^{\gamma_2} 
\ar@<0.5ex>[l]^{\delta_1} & 
}
$$
subject to the relations  that every two paths with the same starting and end point and of same length are equal.
We claim that the algebra $E=B''(3)/\operatorname{rad}^3(B''(3))$ is wild.
To show this we consider the universal cover for $E$, see \cite{G}, and
notice that it has the following subquiver
$$
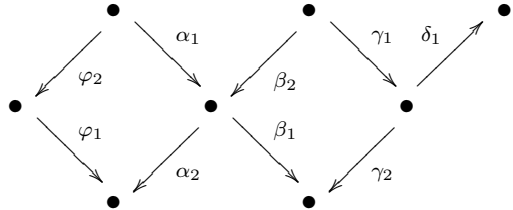
\xymatrix{&\bullet\ar@<0.5ex>[dl]^{\varphi_2}\ar@<0.5ex>[dr]^{\alpha_1}&&\bullet\ar@<0.5ex>[dl]^{\beta_2}\ar@<0.5ex>[dr]^{\gamma_1}&&\bullet\\
\bullet\ar@<0.5ex>[dr]^{\varphi_1}&&\bullet\ar@<0.5ex>[dl]^{\alpha_2}\ar@<0.5ex>[dr]^{\beta_1}&&\bullet\ar@<0.5ex>[dl]^{\gamma_2}\ar@<0.5ex>[ur]^{\delta_1}\\
&\bullet&&\bullet&&& 
}
$$
subject to the relations $\varphi_1\varphi_2=\alpha_2\alpha_1, \beta_1\beta_2=\gamma_2\gamma_1$. This quiver is wild, see for instance 
\cite{U}. Hence $E$ and $B''(3)$ are wild.
\end{proof}

\begin{proposition}\label{wildweight2}  
$A(k)$, $A'(k)$ and $A''(k)$ are wild for any $k\geq 2$, while
$A(1)$ is tame.
\end{proposition}
\begin{proof} First, we consider the
case $k=2$. For $A''(2)$  we take its quotient by
$\alpha=\gamma=\varepsilon=0$, and
for $A'(2)$ we consider the quotient by
$\gamma=\delta=\varepsilon=\sigma=0$.
In both cases we obtain the path algebra of the quiver

$$
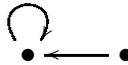
\xymatrix{\bullet  \ar@(ul,ur)& \bullet
\ar[l] }
$$
which is wild.
For $A(2)$ in a similar way one  constructs a quotient algebra
isomorphic to the path algebra of the quiver $
\xymatrix{\bullet \ar@<0.5ex>[r]^{}& \bullet
\ar@<0.5ex>[l]^{} & \ar[l]^{} \bullet}$.

For $k>2$,  $A(k)$, $A'(k)$ and $A''(k)$ always have a quotient algebra isomorphic
to $A'(2)$.  

Finally we notice that the category of nilpotent representations of $A(1)$
is isomorphic to  the category of $\mathbb Z_2$-graded nilpotent
representations of $\mathbb C[\theta]$ with deg$(\theta)=1$. Hence, this category
is tame.
\end{proof}

\subsection{Koszulity}
\begin{proposition}\label{koszul} The algebras $A(k)$ and $B(k)$ for $k\geq 2$ are Koszul.
\end{proposition}
\begin{proof} Let us enumerate the vertices of the quiver $C_k$ by the elements of the group $G_k\simeq \mathbb Z_2^k$ in the natural way. By $\delta_1,\dots,\delta_k$
we denote the basis of $G_k$ over $\mathbb Z_2$. We also consider the inner product $(\cdot,\cdot)$ on $G_k$ such that $(\delta_i,\delta_j)=\delta_{ij}$.

Define the $G_k$-grading of the polynomial algebra $\mathbb C[x_1,\dots,x_k]$ by setting the degree of $x_i$  to equal $\delta_i$. In addition, we consider the 
standard $\mathbb Z$-grading of $\mathbb C[x_1,\dots,x_k]$. Thus, $\mathbb C[x_1,\dots,x_k]$ is now equipped with a $G_k\times\mathbb Z$-grading.
It is not hard to see that the category of all $\mathbb Z$-graded $A(k)$-modules is equivalent to the category of all $G_k\times\mathbb Z$-graded 
$\mathbb C[x_1,\dots,x_k]$-modules. The Koszul resolution of the trivial  $\mathbb C[x_1,\dots,x_k]$-module can be equipped with a $G_k$-grading. Therefore
$A(k)$ is Koszul. 

To prove the Koszulity of $B(k)$, let us first renormalize the arrows of the quiver
$C(k)$ in the following way. If the arrow $\theta$ joins the vertices $v$ and $v+\delta_i$, where $v\in G_k$,
we multiply $\theta$ by $(-1)^{(v,\delta_1+\dots+\delta_i)}$. 
Let $D(k)$ be the quadratic algebra with generators $\xi_1,\dots,\xi_k$ and
relations $\xi_i\xi_j+\xi_j\xi_i=0$, $\xi_i^2=\xi_j^2$ for all $i\neq j$.  Note that $D(k)$ is the dual to the Koszul algebra
$\mathbb C[x_1,\dots,x_k]/(x_1^2+\dots+x_k^2)$ and, hence, $D(k)$ is Koszul. Define a $G_k$-grading on $D(k)$ by setting the degree of $x_i$ to equal $\delta_i$.
Then the category of all graded $B(k)$-modules is equivalent
to the category of $G_k\times\mathbb Z$-graded $D(k)$-modules. Thus,  the Koszulty of $B(k)$ follows by the argument analogous to the one in the previous case.
\end{proof}

Finally let us note that $B'(k)$  for $k\geq 3$ and $B''(k)$  for $k\geq 4$ are quadratic algebras. However, we do not know if they are Koszul.

\section{Generalized weight modules over the Weyl algebra}

The goal of this section is to study the structure of 
$^{\rm b}(\mathcal D,\mathcal H)_\mu-\rm{mod}$.
Note that since the simple modules in $(\mathcal D, \mathcal H)$-mod are
the same as those in  $^{\rm {wb}}(\mathcal D,\mathcal H)$-mod, we can use the
description of simples from \cite{GS}.
We will show that $(\mathcal D, \mathcal H)$-mod has enough injectives and
explicitely construct an injective cogenerator $\mathcal R_\mu$ in
$(\mathcal D,\mathcal H)_\mu-\rm{mod}$.

We use a slight modification of the classical result of
Gabriel. For the proof of this version see \cite{GS2}.

\begin{theorem}\label{antieq} Let  $\mathcal R_\mu$ be an injective
cogenerator of $(\mathcal D,\mathcal H)_\mu-\rm{mod}$, and let 
$\mathcal A_\mu=\operatorname{End}_\mathcal D(\mathcal R_\mu)$. 
The category $\mathcal A_\mu$-fmod of finite dimensional 
$\mathcal A_\mu$-modules
is antiequivalent to the category $^{\rm b}(\mathcal D,\mathcal H)_{\mu}$-mod.  
The mutually inverse contravariant functors which establish this
antiequivalence are ${\rm Hom}_{\mathcal D}(\cdot,\mathcal R_\nu)$
and  ${\rm Hom}_{\mathcal A_\mu}(\cdot,\mathcal R_\nu)$.
\end{theorem}

\subsection{ The case of $\mathcal D(1)$} In this subsection we assume
$\mathcal D=\mathcal D(1)$, $t_0=t$ and $\partial_0=\partial$. 
For $\nu\in\mathbb C$ we set 
$$\mathcal F_\nu=t^\nu\mathbb C[t,t^{-1}]$$
and consider $\mathcal F_\nu$ as a $\mathcal D$-module with the natural action of
$\mathcal D$.  It is easy to check that $\mathcal F_\nu\in{}^{\rm b}(\mathcal D,\mathcal H)_\mu-\rm{mod}$
and  $\mathcal F_\nu$ is simple iff $\nu\notin\mathbb Z$. 
By definition, $\mathcal F_\nu\simeq \mathcal F_\mu$ if and only if
$\mu-\nu\in\mathbb Z$. So, if
$\nu\in\mathbb Z$ we may assume $\nu=0$. It is an easy exercise to
check that $\mathcal F_0$ has length $2$ and one has the following
non-split exact sequence
$$0\to \mathcal F^{+}_0\to\mathcal F_0\to\mathcal F^{-}_0\to 0,$$
where $ \mathcal F^{+}_0=\mathbb C[t]$ and $\mathcal F^{-}_0$ is a
simple quotient. Moreover, if $\sigma$ denotes the automorphism of
$\mathcal D$ defined by $\sigma(t)=\partial, \sigma(\partial)=-t$, then
$\mathcal F^-_0\simeq (\mathcal F^{+}_0)^\sigma$.

As follows for instance from \cite{GS}, any simple object in $(\mathcal D,\mathcal H)$-mod
is isomorphic to $\mathcal F_\nu$ for some non-integer $\nu$,
$\mathcal F^-_0$ or $\mathcal F^+_0$. In this subsection, we will verify that  $(\mathcal D,\mathcal H)$-mod has enough injectives.

Set $u=\log t$ and consider the $\mathcal D$-module 
$$\mathcal F^{\rm log}_\nu=\mathcal F_{\nu}\otimes\mathbb C[u].$$
One can easily check that $\mathcal F^{\rm log}_\nu\in (\mathcal D,\mathcal H)_\nu$-mod
and that  $\mathcal F^{\rm log}_\nu$ has a unique simple submodule
isomorphic to $\mathcal F_\nu$ for $\nu\notin\mathbb Z$ and to $\mathcal F_0^+$ for $\nu\in\mathbb Z$.

\begin{proposition}\label{injective1}  The module ${\mathcal F}^{\log}_\nu$ is an
injective module in the category  $(\mathcal D,\mathcal H)$-mod.
\end{proposition}
\begin{proof} If $\nu\in\mathbb Z$ we assume $\nu=0$. Let $N={\mathcal F}^{\log}_\nu$. For any  $\mu\in\operatorname{supp}N$
the generalized weight subspace $N^{(\mu)}$ is isomorphic to the
direct limit 
$$\lim_{\longrightarrow}\mathcal H/(t\partial-\mu)^m.$$
Hence $N^{(\mu)}$ is injective in the category of locally finite
$\mathcal H$-modules.

For any $\mathcal D$-module $M$ let $\Gamma_{\mathcal H}(M)$ be the
set of all finite $\mathcal H$-vectors. Note that $\Gamma_{\mathcal H}(M)$ is 
in fact $\mathcal D$-invariant and therefore 
$$\Gamma_{\mathcal H} : \mathcal D-{\rm mod}\to (\mathcal D,\mathcal H)-{\rm mod}$$
is a functor right adjoint to the embedding functor $ (\mathcal D,\mathcal H)-{\rm mod}\to\mathcal D-{\rm mod}$.

We claim now that 
$$I_\nu=\Gamma_{\mathcal H}({\rm Hom}_\mathcal H(\mathcal D,N^{(\nu)}))$$
is injective in $(\mathcal D,\mathcal H)$-mod. Indeed, for any $M\in (\mathcal D,\mathcal H)$-mod
$$\operatorname{Hom}_{\mathcal D}(M,I_\nu)= \operatorname{Hom}_{\mathcal D}(M,{\rm Hom}_\mathcal H(\mathcal D,N^{(\nu)}))=
 \operatorname{Hom}_{\mathcal H}(M,N^{(\nu)}),$$
where the second equality follows from the Frobenius reciprocity.
Thus, ${\rm Hom}_\mathcal D(\cdot, I_\nu)\simeq  \operatorname{Hom}_{\mathcal H}(\cdot,N^{(\nu)})$ is an exact functor.

It remains to prove  that $N$ is isomorphic
to $I_\nu$. Consider the homomorphism $\varphi: N\to I_\nu$ induced by
the projection $N\to N^{(\nu)}$ via Frobenius reciprocity.  
Then $\varphi$ is injective since it is not zero on the unique simple
submodule of $N$. To prove the surjectivity of $\varphi$ note that any
generalized weight subspace $I_\nu^{(\mu)}$ is isomorphic to $N^{(\mu)}$
and any non-zero $\mathcal H$-linear map  $I_\nu^{(\mu)}\to N^{(\mu)}$ is surjective.
\end{proof}

It is clear from above that we have constructed injective hulls of all up to isomorphism simple
objects in $(\mathcal D,\mathcal H)$-mod except of  $\mathcal F_0^-$. To
construct an indecomposable injective hull of $\mathcal F_0^-$ we 
use the twist by $\sigma$. Indeed, $(\mathcal F_0^{\rm log})^\sigma$ is
injective with unique simple submodule 
$(\mathcal F_0^+)^\sigma\simeq \mathcal F_0^-$.

Thus, we have constructed an injective cogenerator $\mathcal R_\nu$ for every block 
$(\mathcal D,\mathcal H)_\nu-\rm{mod}$. Namely, if $\nu\notin\mathbb Z$ we
have   $\mathcal R_\nu=\mathcal F^{\rm log}_\nu$, and
if $\nu=0$ we have $\mathcal R_0= \mathcal F^{\rm log}_0\oplus(\mathcal F^{\rm log}_0)^\sigma.$

\begin{lemma}\label{end1} $\operatorname{End}_\mathcal D(\mathcal F^{\rm log}_\nu)\simeq \mathbb C[[z]]$.
\end{lemma}
\begin{proof} First, we consider the embedding $j: C[[z]]\to \operatorname{End}_\mathcal D(\mathcal F^{\rm log}_\nu)$ defined by $z \mapsto \frac{\partial}{\partial u}$. It remains to prove  that $j$ is 
surjective. Note that $\mathcal F^{\rm log}_\nu$ is equipped
with an increasing exhausting filtration $$0=F_0\subset F_1\subset F_2\subset\dots,$$
with $F_1=\mathcal F_\nu$ and $F_k=F_{k-1}+uF_{k-1}$. We have 
\begin{equation}\label{eqend1}
\operatorname{Hom}_{\mathcal D}(F_k/F_{k-1},F_1)\simeq \mathbb C.
\end{equation}
Every $f\in \operatorname{End}_\mathcal D(\mathcal F^{\rm log}_\nu)$
preserves the filtration and hence we have
$$\operatorname{End}_\mathcal D(\mathcal F^{\rm log}_\nu)=\lim_{\longleftarrow}\operatorname{End}_\mathcal D(F_k).$$ 
It follows easily from (\ref{eqend1}) that $\operatorname{End}_\mathcal D(F_k)\simeq \mathbb C[z]/(z^{k+1})$. This completes the proof.
\end{proof}

\begin{corollary}\label{end2}  $\operatorname{End}_\mathcal D((\mathcal F^{\rm log}_0)^\sigma)\simeq \mathbb C[[z]]$.
\end{corollary}

\begin{lemma}\label{end2} There exists
$\theta\in \operatorname{End}_\mathcal D(\mathcal R_0)$ such that 
$\theta(\mathcal F_0^{\rm log})= (\mathcal F_0^{\rm log})^\sigma$,
$\theta((\mathcal F_0^{\rm log})^\sigma)=\mathcal F_0^{\rm log}$ 
and $\theta^2=z$.
\end{lemma}
\begin{proof} Recall that $\mathcal F_0/\mathcal F_0^+\simeq \mathcal F_0^-$, and this isomorphism extends to 
homomorphism $\theta^+:\mathcal F_0^{\rm log}\to (\mathcal F_0^{\rm log})^\sigma$ with kernel 
$\mathcal F_0^+$. By twisting with $\sigma$ we construct 
$\theta^-:(\mathcal F_0^{\rm log})^\sigma\to \mathcal F_0^{\rm log}$  with kernel 
$\mathcal F_0^-$. By construction
$$\operatorname{Ker}(\theta^-\circ\theta^+)=\mathcal F_0$$
and
$$\operatorname{Ker}(\theta^+\circ\theta^-)=\mathcal F_0^{\sigma}.$$

That implies $\theta^-\circ\theta^+=z f$ for some unit $f\in\mathbb C[z]$. 
By normalizing we can assume that   $f=1$. 
Finally, set $\theta=\theta^++\theta^-$.
\end{proof}

\begin{corollary}\label{end3} The algebra $\operatorname{End}_\mathcal D(\mathcal R_0)$
is generated by $\mathbb C[[\theta]]$ and the idempotents $e^+$ and
$e^-$ representing projectors on $\mathcal F_0^{\rm log}$ and  $(\mathcal F_0^{\rm log})^\sigma$, respectively. 
\end{corollary}

\begin{theorem}\label{eq1} Let $\nu\in\mathbb C$.

(a) If $\nu\notin \mathbb Z$, then  $^{\rm b}(\mathcal D,\mathcal H)_{\nu}$-mod 
is equivalent to the category of finite-dimensional $\mathbb C[z]$-modules with nilpotent 
action of $z$.

(b) If $\nu\in\mathbb Z$, then $^{\rm b}(\mathcal D,\mathcal H)_{\nu}$-mod 
is equivalent to the category of finite-dimensional $\mathbb Z_2$-graded $\mathbb C[\theta]$-modules  with nilpotent action of $\theta$, where the $\mathbb Z_2$-grading on
$\mathbb C[\theta]$ is defined by $\deg \theta = 1$.
\end{theorem} 

\begin{proof} The theorem is an easy consequence of Theorem \ref{antieq},
Lemma \ref{end1}, and Corollary \ref{end3}.

Indeed, in the first case $\mathcal A_\nu=\mathbb C[[z]]$ and the
category $\mathcal A_\nu$-fmod is equivalent to the category  finite-dimensional $\mathbb C[z]$-modules with nilpotent 
action of $z$. Moreover, the latter category clearly has a nice
duality functor. Hence, we can change the  antiequivalence by an  equivalence.

In the second case  $\mathcal A_\nu$ is generated by $\mathbb C[[\theta]]$ and two idempotents.  
Clearly $\theta$ acts nilpotently on all finite-dimensional  $\mathbb C[[\theta]]$-modules. The
idempotents correspond to projectors on the homogeneous parts of the
$\mathbb Z_2$-grading.
\end{proof}

Observe that, by the above theorem,  $^{\rm b}(\mathcal D,\mathcal H)_{\nu}$-mod is tame.
If $\nu\notin\mathbb Z$, all indecomposable representations in this category are
parametrized by their dimension that can be any positive integer. If $\nu\in\mathbb Z$, the indecomposable
representations are enumerated by their superdimension which can take values
$(n+1,n),(n,n+1)$ and $(n,n)$. In the latter case there are two up to
isomorphism indecomposable representations with the same superdimension, and
one is obtained from the other by changing the parity.

\subsection{The case of ${\mathcal D} (n+1)$ for arbitrary $n$} \label{subsec-dn}
We assume now $\mathcal D=\mathcal D(n+1)$ and repeatedly use the
fact that $\mathcal D(n+1)$ is the tensor product of $n+1$ copies of
$\mathcal D(1)$. We set ${\mathcal F}_{\nu}^{\log} = {\mathcal F}_{\nu_0}^{\log} \otimes... \otimes {\mathcal F}_{\nu_n}^{\log}$ and ${\mathcal F}_{\nu}= {\mathcal F}_{\nu_0} \otimes... \otimes {\mathcal F}_{\nu_n}$. 

It follows from \cite{GS} that for $\nu=(\nu_0,\dots,\nu_n)$
every simple module in  $(\mathcal D,\mathcal H)_\nu$-mod is the tensor
product $S_0\otimes S_1\otimes\dots\otimes S_n$ where each $S_i$ is a
simple module in $(\mathcal D(1),\mathcal H(1))_{\nu_i}$-mod.

Let ${\rm Int} (\nu)$ be the set of all $i$ such that $\nu_i \in \mathbb Z$ and $\mathcal P (\nu)$ be the power set of ${\rm Int}(\nu)$.
For every $J\in\mathcal P(\nu)$ set
$$\mathcal S_\nu(J):=S_0\otimes S_1\otimes\dots\otimes S_n,$$
where $S_i=\mathcal F_{\nu_i}$ if $i\notin {\rm Int}(\nu)$,  $S_i=\mathcal F^+_{0}$
if $i\in {\rm Int}(\nu)\setminus J$ and $S_i=\mathcal F^-_{0}$
if $i\in J$. 

\begin{lemma}\label{simpleD} Any simple object in  $(\mathcal D,\mathcal H)_\nu$-mod
is isomorphic to $\mathcal S_\nu(J)$ for  $J\in\mathcal P(\nu)$. Moreover, 
$\mathcal S_\nu(J)\simeq \mathcal S_\nu(I)$
if and only if $I=J$. If by $\sigma_i$ we denote the automorphsim of
$\mathcal D$ induced by $\sigma$ on the $i$-th copy of $\mathcal D(1)$
and set $\sigma_J=\prod_{i\in J}\sigma_i$, then
$\mathcal S_\nu(J)=(\mathcal S_\nu(\emptyset))^{\sigma_J}$.
\end{lemma}
\begin{proof} The lemma follows from \cite{GS}. 
\end{proof}

Similarly as above, for any $J\in\mathcal P(\nu)$ set
$$\mathcal I_\nu(J):=I_0\otimes I_1\otimes\dots\otimes I_n,$$
where $I_j=\mathcal F_{\nu_j}^{\rm log}$ if $j\notin {\rm Int}(\nu)$,
$I_j=\mathcal F_{0}^{\rm log}$ if
$j\in {\rm Int}(\nu)\setminus J$ and $I_j=(\mathcal F_{0}^{\rm log})^{\sigma}$
if $j\in J$. 
Observe that 
$$\mathcal I_\nu(\emptyset)=\prod_{i=0}^n t_i^{\nu_i}\mathbb C[t_0^{\pm 1},\dots,t_n^{\pm 1},u_0,\dots,u_n]$$
with $u_i={\rm log}(t_i)$. We  also set ${\mathcal F}_{\nu}^{\log} (J) = ({\mathcal F}_{\nu}^{\log})^{\sigma_J}$ and ${\mathcal F}_{\nu} (J) = ({\mathcal F}_{\nu})^{\sigma_J}$

\begin{lemma}\label{inj4} 

(a) For any $J\in\mathcal P(\nu)$, $\mathcal I_\nu(J)$ is the
indecomposable injective hull of $\mathcal S_\nu(J)$ in the category  $(\mathcal D,\mathcal H)$-mod.

(b) $\mathcal R_\nu=\bigoplus_{J\in\mathcal P(\nu)}\mathcal I_\nu(J)$ 
is an injective cogenerator in $(\mathcal D,\mathcal H)_\nu$-mod.

\end{lemma}
\begin{proof}
(a) The proof in the case $J=\emptyset$ is the same as in the case $n=1$,
see Proposition \ref{injective1}. For $J \neq \emptyset$, we use the twist by $\sigma_J$.

Part (b)  follows  from (a).
\end{proof}

Let
$\mathcal A_\nu={\rm End}_\mathcal D(\mathcal R_\nu)$.
Recall that $\mathcal R_\nu=\mathcal R_{\nu_0}\otimes\dots\otimes \mathcal R_{\nu_n}$.
Define a $\mathbb Z^{n+1}$-filtration on $\mathcal R_\nu$ by 
$$F_{m_0,\dots,m_n}(\mathcal R_{\nu})=F_{m_0}(\mathcal R_{\nu_0})\otimes\dots\otimes F_{m_n}(\mathcal R_{\nu_n}).$$ 
Again, it is not hard to see that this filtration is preserved by
any $f\in A_\nu$. For each ${\bold m}\in\mathbb Z^{n+1}$ set 
$$A_\nu^{\bold m}={\rm End}_\mathcal D(F_{\bold m}(\mathcal R_\nu)).$$
Then 
$$\mathcal A_\nu=\lim_{\longleftarrow}\mathcal A_\nu^{\bold m},$$
and 
$$\mathcal A_\nu^{\bold m}=\operatorname{End}_{\mathcal
D(1)}(F_{m_0}(\mathcal R_{\nu_0}))\otimes\dots\otimes \operatorname{End}_{\mathcal
D(1)}(F_{m_n}(\mathcal R_{\nu_n})).$$

For any $i\notin{\rm Int}(\nu)$ let $z_i\in\mathcal A_\nu$ be the
endomorphism induced by $z$ on $\mathcal R_{\nu_i}$. If $i\in{\rm Int}(\nu)$
let $\theta_i,e_i^+,e_i^-$ be induced by the corresponding endomorphisms
of  $\mathcal R_{\nu_i}$.

The proof of the following is straightforward.

\begin{lemma}\label{endgen}  $\mathcal A_\nu$  is generated by
$\mathbb C[[z_i,\theta_j]]_{i\notin{\rm Int}(\nu),j\in{\rm Int}(\nu)}$
and $e_i^\pm$ for all  $i\in{\rm Int}(\nu)$. 
\end{lemma}

Now applying Theorem \ref{antieq} leads to  the following result 

\begin{proposition}\label{equivgen} Assume that $|{\rm Int}(\nu)|=k$, and let
$$G_k=\mathbb Z_2^k=\bigoplus_{i=1}^k \mathbb Z_2 \delta_i.$$ Consider the 
$G_k$-grading on $\mathbb C[\theta_1,...,\theta_k,z_1,\dots,z_{n+1-k}]$ defined by $\deg (z_i)=0$
and $\deg (\theta_j)=\delta_j$.
Then  $^{\rm b}(\mathcal D,\mathcal H)_\nu$-mod
is equivalent to the category of 
finite-dimensional $G_k$-graded $\mathbb C[\theta_1,...,\theta_k,z_1,\dots,z_{n+1-k}]$-modules with nilpotent
action of all  $z_i$ and $\theta_j$.
\end{proposition}

Note that for $^{\rm b}(\mathcal D,\mathcal H)_\nu$-mod is wild for $n>0$.

\subsection{The category ${}^{\rm b}_{\rm s}(\mathcal D, \mathcal H)$-mod}\label{subsec-d-weight} 

Recall the definition of $^{\rm b}_{\rm s}(\mathcal D, \mathcal H)$-mod in Section \ref{subsec-wtde}.

Define the  left exact functor
$S_{\mathcal H'}: (\mathcal D, \mathcal H)$-mod $\to {} _{\rm s}(\mathcal D, \mathcal H)$-mod to be the one that  maps $M$ to its submodule consisting of all $\mathcal H'$-eigenvectors. By
 general nonsense arguments, $S_{\mathcal H'}$ maps injectives to
injectives and blocks to blocks. Therefore, $S_{\mathcal H'}(\mathcal R_\nu)$
is an injective cogenerator 
in the block $_{\rm s}(\mathcal D, \mathcal H)_\nu$-mod.
It is not hard to see that 
$$S_{\mathcal H'}(\mathcal R_\nu)=\mathcal F_\nu\otimes\mathbb C[u],$$
where $u=\log (t_0t_1\cdots t_n)$.
If we set $\mathcal B_\nu=\operatorname{End}(S_{\mathcal H'}(\mathcal R_\nu))$ 
then clearly  $\mathcal B_\nu$ is the quotient of  $\mathcal A_\nu$ by
the ideal generated by $\frac{\partial}{\partial u_i}-\frac{\partial}{\partial u_i}$
for all $i\neq j$.
This implies the following.

\begin{lemma} \label{weightalgebra} For any  $\nu$, the ring
  $\mathcal B_\nu$ is the $\mathbb C[[z]]$-algebra  generated by 
  $\theta_i$ and the  idempotents $e^\pm_i$ for all $i\in\rm{Int}(\nu)$,  subject to the relations $\theta_i^2=z$. 
\end{lemma}

Using  Theorem \ref{antieq} again, we obtain the following result.

\begin{proposition}\label{equivgenweight} Assume 
that $|{\rm    Int}(\nu)|=k>0$. 
Consider the $G_k$-grading on ${\mathcal B} (k)= \mathbb C[\theta_1,...,\theta_k]/(\theta_i^2-\theta_j^2)$ defined by 
$\deg (\theta_i)=\delta_i$.
Then  $_{\rm s}^{\rm b}(\mathcal D,\mathcal H)_\nu$-mod
is equivalent to the category of 
finite-dimensional $G_k$-graded ${\mathcal B} (k)$-modules  with nilpotent
action of all $\theta_i$.
If $k=0$, then  $_{\rm s}^{\rm b}(\mathcal D,\mathcal H)_\nu$-mod
is equivalent to the category of finite-dimensional $\mathbb C[z]$-modules 
with nilpotent action of $z$.
\end{proposition}

\begin{corollary}\label{wildweight}  
The category $_{\rm s}^{\rm b}(\mathcal D,\mathcal H)_\nu$-mod is tame if and
only if $|{\rm    Int}(\nu)|=0,1$. \end{corollary}

%%%%%%%%%%%%%
\section{$\mathcal D^E$ and twisted differential operators on $\mathbb P^n$} In this section we assume that $n \geq 1$.
\subsection{Connection between weight $\mathcal D$-modules and weight  $\mathcal D^E$-modules} \label{subsec-de}

For any $M\in  (\mathcal D, \mathcal H)$-mod, let 
$$\Gamma_a(M)=\bigcup_{l>0}\operatorname{Ker}(E-a)^l.$$
It is not hard to see that $\Gamma_a$ is an exact functor from the
category $(\mathcal D, \mathcal H)$-mod to the category 
$(\mathcal D^E, \mathcal H)^{a}$-mod.
The induction functor
$$\Phi(X)=\mathcal D\otimes_{\mathcal D^E} X$$
is its left adjoint.

\begin{lemma}\label{twisted1}
(a) $\Gamma_a\circ\Phi$ is isomorphic to the identity functor.

(b) If $S$ is simple in $(\mathcal D, \mathcal H)$-mod, then
$\Gamma_a(S)$ is either simple or zero. 

(c) Let $S_1$ and $S_2$ be
non-isomorphic simple modules in  $(\mathcal D, \mathcal H)$-mod
and $\Gamma_a(S_i)\neq 0$ for $i=1,2$. Then $\Gamma_a(S_1)$ and 
$\Gamma_a(S_2)$ are not isomorphic.

(d) Any simple $M$ in $(\mathcal D^E, \mathcal H)^{a}$-mod
is isomorphic to $\Gamma_a(S)$ for some simple $S$ in 
$(\mathcal D, \mathcal H)$-mod.

(e) If $I$ is the indecomposable injective hull of a simple $S$ in 
$(\mathcal D, \mathcal H)$-mod and $\Gamma_a(S)\neq 0$, then
$\Gamma_a(I)$ is injective in $(\mathcal D^E, \mathcal H)^{a}$-mod.  
\end{lemma}
\begin{proof} (a) Use
$$\Phi(X)\simeq\bigoplus_{m\in\mathbb Z}\mathcal D^m\otimes_{\mathcal D^E} X,$$ 
and
$$\mathcal D^m\otimes_{\mathcal D^E} X=\bigcup_{l>0}\operatorname{Ker}(E-a-m)^l.$$
Since $\mathcal D^0=\mathcal D^E$, we have  $$\Gamma_a(X)=\mathcal D^0\otimes_{\mathcal D^E} X\simeq X.$$ 

(b) Note that both $S$ and $\Gamma_a(S)$ are multiplicity free weight
modules.
If $N$ is a proper non-zero submodule in $\Gamma_a(S)$, then ${\rm supp} N$ is
a proper non-empty subset in ${\rm supp} \Gamma_a(S)$. Let $S'$ be the submodule
in $S$ generated by $N$. Then  ${\rm supp}\, S'$ is a proper non-empty
subset in ${\rm supp} S$. That contradicts to the simplicity of $S$.

(c) The statement follows from the fact that $S_1$ and $S_2$ have disjoint supports.

(d) By (a), we have $M\simeq \Gamma_a(\Phi(M))$. Hence
$M\simeq \Gamma_a(S)$ for some simple subquotient $S$ of $\Phi(M)$.

(e) Let $\varphi: X\to Y$ be an injective homomorphism of modules in  
$(\mathcal D^E, \mathcal H)^{a}$-mod, and $\eta:X\to \Gamma_a(I)$ be
some homomorphism. Let $Z$ be the kernel of
$\Phi(\varphi):\Phi(X)\to\Phi(Y)$, and $\eta':\Phi(X)\to I$ be the
homomorphism induced by $\eta$. So we have $\Gamma_a(\eta')=\eta$. 
Note that by (a) we have $\Gamma_a(Z)=0$. We claim that
$\eta'(Z)=0$. Indeed, otherwise $\eta'(Z)$ contains
$S$, hence $Z$ has a subquotient isomorphic to $S$. But
$\Gamma_a(S)\neq 0$. 
Then by the injectivity of $I$, there exists a homomorphism
$\alpha:\Phi(Y)\to I$ such that $\eta'=\alpha\circ\Phi(\varphi)$.
After applying $\Gamma_a$ to the latter identity we obtain
$\eta=\beta\circ\varphi$ with $\beta=\Gamma_a(\alpha)$.
\end{proof}

Note that if  $\nu\in\mathbb C^{n+1}$, then the restriction of $\Gamma_a$ to 
$(\mathcal D, \mathcal H)_{\nu}$-mod is not identically zero
if and only if $|\nu|-a\in \mathbb Z$. 

\begin{lemma}\label{twistednonint} Let $\nu\notin \mathbb Z^{n+1}$
and $a-|\nu|\in \mathbb Z$. Then for any simple $S$ in $(\mathcal D, \mathcal H)_{\nu}$-mod 
$\Gamma_a(S)\neq 0$.
\end{lemma}
\begin{proof} We just have to check that the intersection of 
${\rm supp} S$ with the hyperplane $|\mu|=a$ is not empty.
This is an immediate consequence of
\begin{equation}\label{support}
{\rm supp}(\mathcal S_\nu(J))=\sum_{i\in J}\mathbb
Z_{<0}\varepsilon_i+
\sum_{i\in {\rm Int}(\nu)\setminus J}\mathbb Z_{\geq 0}\varepsilon_i+\sum_{i\notin {\rm Int}(\nu)}(\nu_i+\mathbb Z)\varepsilon_i.
\end{equation}
\end{proof}
\begin{corollary}\label{twistednoninteq}  Let $\nu\notin \mathbb Z^{n+1}$
and $a-|\nu|\in \mathbb Z$. Then $(\mathcal D, \mathcal H)_{\nu}$-mod
and $(\mathcal D^E, \mathcal H)^a_{\nu}$-mod are equivalent.
\end{corollary}
\begin{proof} From the previous lemma we have that $\Phi$ and
$\Gamma_a$ are inverse.
\end{proof}
\subsection{ On the structure of $ (\mathcal D^E, \mathcal H)^a_{0}$-mod} \label{subsec-deha}
We now consider the case  $\nu\in  \mathbb Z^{n+1}$. Without loss
of generality we may assume $\nu=0$.
\begin{lemma}\label{twistedint} Let $a\in\mathbb Z$.

(a) Let  $-n-1<a<0$ and  $\Gamma_a(\mathcal S_0(J))=0$ if and only if 
$J=\emptyset$ or $J=\{0,\dots,n+1\}$.

(b) If $a\geq 0$ and  $\Gamma_a(\mathcal S_0(J))=0$  if and only if 
$J=\{0,\dots,n+1\}$.

(c) If $a\leq -n-1$ and  $\Gamma_a(\mathcal S_0(J))=0$  if and only if 
$J=\emptyset$. 
\end{lemma}
\begin{proof}  Like the proof of Lemma \ref{twistednonint}, the statements follow by direct inspection using (\ref{support}).
\end{proof}

Consider the automorphism  $\tau=\prod_{i=0}^n\sigma_i$ of $\mathcal D$. It is clear that 
$\tau(E)=-E-n-1$. Hence $\tau(\mathcal D^E)=\mathcal D^E$ and twist by
$\tau$ induces an equivalence between $(\mathcal D^E,\mathcal H)_a$-mod and 
$(\mathcal D^E,\mathcal H)_{-n-1-a}$-mod. In particular,  $ (\mathcal D^E, \mathcal H)^a_{0}$-mod 
and  $ (\mathcal D^E, \mathcal H)^{-n-a-1}_{0}$-mod are equivalent. Thus, without
loss of generality, we can
restrict our attention to the cases (a) and (b) of the last lemma.

If $a\geq 0$, then all up to isomorphism simple
modules of   $ (\mathcal D^E, \mathcal H)^a_{0}$-mod are of the form
$\mathcal S^a_0(J)=\Gamma_a(\mathcal S_0(J))$, where $J$ runs through all  proper  
subsets of $\{0,\dots,n\}$. If  $-n-1<a<0$, then  
all up to isomorphism simple
modules of   $ (\mathcal D^E, \mathcal H)^a_{0}$-mod are of the form
$\mathcal S^a_0(J)=\Gamma_a(\mathcal S_0(J))$, where $J$ runs through all  proper non-empty subsets of 
$\{0,\dots,n\}$. In what follows we denote by
$\mathcal P_a$ the collection of all proper subsets of $\{0,\dots,n\}$
for $a\geq 0$, and that of all proper non-empty subsets of $\{0,\dots,n\}$
for $-n-1<a<0$.

Set
$$\mathcal R^a_0=\bigoplus_{J\in\mathcal P_a} \Gamma_a(\mathcal I_0(J)).$$
Lemma \ref{twisted1} implies the following.
\begin{corollary}\label{injco} $\mathcal R^a_0$ is an injective
cogenerator in  $ (\mathcal D^E, \mathcal H)^a_{0}$-mod.
\end{corollary}

In view of Theorem \ref{antieq}, it is useful to study the algebra
$\operatorname{End}_{\mathcal D^E}(\mathcal R^a_0)$.
Let $\pi^a\in \mathcal A_0$ denote the projector onto
$\bigoplus_{J\in\mathcal P_a} \mathcal I_0(J)$.

\begin{lemma}\label{idempotent} 
$\operatorname{End}_{\mathcal D^E}(\mathcal R^a_0)$ is isomorphic to
$\pi^a\mathcal A_0\pi^a$.
\end{lemma}
\begin{proof} Observe that 
$$\pi^a\mathcal A_0\pi^a=\operatorname{End}_{\mathcal D}(\bigoplus_{J\in\mathcal P_a} \mathcal I_0(J)).$$ 
We define a homomorphism
$\gamma_a:\pi^a\mathcal A_0\pi^a\to \operatorname{End}_{\mathcal D^E}(\mathcal R^a_0)$ 
by setting $\gamma_a(g)=\Gamma_a(g)$.
The injectivity of $\gamma_a$ follows immediately by definition.
To prove the surjectivity consider $f\in\operatorname{End}_{\mathcal D^E}(\mathcal R^a_0)$.
Note that $\Gamma_a$ has a right adjoint functor $\Phi'$ which has all
properties analogous to 
the left adjoint functor $\Phi$. It can be defined by
$$\Phi'(M)=\Gamma_{\mathcal H}(\operatorname{Hom}_{\mathcal D^E}(\mathcal D,M)).$$
By general nonsense arguments, $\Phi'$ maps an injective module to an injective
module and therefore 
$$\Phi'(\mathcal R^a_0)=\bigoplus_{J\in\mathcal P_a} \mathcal I_0(J).$$
Then we have $f=\gamma_a(\Phi'(f))$.
\end{proof}

In order to give a combinatorial description of  
$^{\rm b} (\mathcal D^E, \mathcal H)^a_{0}$-mod
we study certain quivers. 
By Proposition \ref{equivgen},   $ (\mathcal D, \mathcal H)_{0}$-mod is equivalent to the
category of finite-dimensional nilpotent representations of  $A(n+1)$,
where by ``nilpotent'' we mean that every sufficiently long path acts trivially.

The following statement is now a consequence of Lemma \ref{idempotent}.

\begin{proposition}\label{singularblocks}

(a) If $a\geq 0$, then  $^{\rm b} (\mathcal D^E, \mathcal H)^a_{0}$-mod is
equivalent to the category of nilpotent
representations of $A'(n+1)$.

(b) If  $-n-1<a<0$, then  
$^{\rm b} (\mathcal D^E, \mathcal H)^a_{0}$-mod is
equivalent to the category of  nilpotent
representations of $A''(n+1)$.

\end{proposition}

\begin{corollary}\label{wildtwisted} For any $n\geq 1$, the category
$^{\rm b} (\mathcal D^E, \mathcal H)^a_{0}$-mod is wild.
\end{corollary}

\subsection{On the structure of $ _{\rm s}(\mathcal D^E, \mathcal H)^a_{0}$-mod}
Recall the definition of $_{\rm s}(\mathcal D^E, \mathcal H)$ in Section \ref{subsec-wtde}. We can use the functor $\Gamma_a$ in the same way as in Section \ref{subsec-de}.
\begin{proposition}\label{twistedweight1} Let $\nu\notin \mathbb Z^{n+1}$
and $a-|\nu|\in \mathbb Z$. Then $_{\rm s}(\mathcal D, \mathcal H)_{\nu}$-mod
and $_{\rm s}(\mathcal D^E, \mathcal H)^a_{\nu}$-mod are equivalent.
\end{proposition}
\begin{proof} The statement follows immediately from the analogue of
Lemma \ref{twisted1}.
\end{proof}
Now we concentrate on the case   $\nu\in \mathbb Z^{n+1}$ and, as
before,  assume that $\nu=0$.
We apply the same arguments as in Section \ref{subsec-de}. We denote the idempotent of $\mathcal B(n+1)$ again by $\pi^a$. Recall that  by Proposition 
\ref{equivgenweight},
$^{\rm b}_{\rm s}(\mathcal D^E, \mathcal H)^a_{0}$-mod is equivalent
to the category of nilpotent $\mathcal B(n+1)$-modules. With this in mind, we have the following.

\begin{lemma}\label{twistedweight2}
The category $^{\rm b}_{\rm s}(\mathcal D^E, \mathcal H)^a_{0}$-mod is equivalent
to the category of nilpotent $\pi^a\mathcal B(n+1)\pi^a$-modules.
\end{lemma}

We now formulate the above result in terms of quivers.

\begin{proposition}\label{singularblocksweights}

(a) If $a\geq 0$, then  $_{\rm s}^{\rm b} (\mathcal D^E, \mathcal H)^a_{0}$-mod is
equivalent to the category of nilpotent
representations of $B'(n+1)$. Therefore, $_{\rm s}^{\rm b} (\mathcal D^E, \mathcal H)^a_{0}$-mod is tame if and only if $n=1$. 

(b) If $-n-1<a<0$, then  
$_{\rm s}^{\rm b} (\mathcal D^E, \mathcal H)^a_{0}$-mod is
equivalent to the category of  nilpotent
representations of $B''(n+1)$. Therefore, $_{\rm s}^{\rm b} (\mathcal D^E, \mathcal H)^a_{0}$-mod is tame if and only if $n=1$.
\end{proposition}

\section{Twisted Localization}

\subsection{Twisted localization in general setting}\label{subsec-loc-gen} Retain the notation of  \S \ref{subsec-weight}. Namely, ${\mathcal U}$ is a finitely generated associative unital  algebra, and ${\mathcal H} = {\mathbb C}[h_0,...,h_n] = S({\mathfrak h})$ is a subalgebra of ${\mathcal U}$ such that ${\rm ad} (h)$ is semisimple on ${\mathcal U}$ for every $h \in {\mathfrak h}$. Let now $F = \{ f_1,...,f_k \}$ be a subset of commuting elements of ${\mathcal U}$ such that 
$\mbox{ad}\, (f_i)$ are locally nilpotent endomorphisms of ${\mathcal U}$. Let 
$\langle F \rangle$ be the multiplicative subset  of ${\mathcal U}$ generated by $\{f_1,...,f_k\}$, i.e. the $\langle F \rangle$  consists of the elements $f_1^{k_1}...f_k^{n_k}$ for $n_i \in {\mathbb Z}_{\geq 0}$. By $D_{F} {\mathcal U}$ we denote the localization of ${\mathcal U}$ relative to $\langle F \rangle$. Note that $\langle F \rangle$ satisfy Ore's localizability condition due to the fact that $f_i$ are locally ad-nilpotent. For a ${\mathcal U}$-module $M$, by $D_F M = D_F {\mathcal U}\otimes_{\mathcal U} M$ we denote the localization of $M$ relative to $\langle F \rangle$. We will consider $D_F M$ both as a ${\mathcal U}$-module and a $D_F {\mathcal U}$-module. 
By  $\theta_{F} : M \to {D}_{F} M$ we denote the localization map defined by $\theta_{F} (m) = 1 \otimes m $. Then
$$
\mbox{ann}_M F:= \{ m \in M \; | \; sm = 0 \mbox{ for some }s \in F\}
$$
is a submodule of $M$  (often called, the torsion submodule with respect to $F$). Note that if $\mbox{ann}_M F = 0$, then  $\theta_{F}$ is an injection. In the latter case, $M$ will be considered naturally as a submodule of ${D}_{F} M$.   Note also that if $F = F_1 \cup F_2$, then $D_{F_1}D_{F_2} M  \simeq D_{F_2}D_{F_1} M \simeq D_F M$.

It is well known that ${D}_{F}$ is a functor from the category of ${\mathcal U}$-modules to the category of ${D}_{F} {\mathcal U}$-modules. For any category ${\mathcal A}$ of ${\mathcal U}$-modules, by $ {\mathcal A}_{F}$ we denote the category of ${D}_{F} {\mathcal U}$-modules that considered as ${\mathcal U}$-modules are in ${\mathcal A}$.  Some useful properties of the localization functor ${D}_{F}$ are listed in the following lemma.

\begin{lemma} \label{lem-loc-map}
(i) If $\varphi : M \to N$ is a homomorphism of ${\mathcal U}$-modules, then ${D}_{F} (\varphi) \theta_{F} = \theta_{F} \varphi$.

(ii) ${D}_{F}$ is an exact functor.

(iii) (Universal property) If $N$ is a ${D}_{F} {\mathcal U}$-module and $\varphi : M \to N$ is a homomorphism of ${\mathcal U}$-modules, then there exists a unique homomorphism of ${D}_{F} {\mathcal U}$-modules $\overline{\varphi}: {D}_{F} M \to N$ such that $\overline{\varphi} \theta_{F} = \varphi$. If we identify $N$ with ${D}_{F} N$, then $\overline{\varphi} = {D}_{F}(\varphi)$.

(iv)  Let ${\mathcal A}$ be any category of $U$-modules. If $I$ is an injective module in ${\mathcal A}_{F}$, then $I$ (considered as an ${\mathcal U}$-module) is injective in ${\mathcal A}$ as well.
\end{lemma}
\begin{proof}
Statements (i)-(iii) are standard (see for example \cite{MR}). For (iv), consider a morphism $\varphi : A \to I$ and a monomorphism  $\psi: A \to B$ in  ${\mathcal A}$. By (ii), ${D}_{F} \psi$ is a monomorphism, and since $I$ is injective in $ {\mathcal A}_{F}$, there is a morphism $\gamma : {D}_{F} A \to I$ such that $\gamma {D}_{F} (\psi) = {D}_{F} (\varphi)$. But then $\alpha : = \gamma \theta_{F}$ is a morphism from $B$ to $I$ and 
$$
\alpha \psi = \gamma \theta_{F}\psi = \gamma {D}_{F} (\psi) \theta_{F} = {D}_{F} (\varphi) \theta_{F} = \varphi.
$$
The second identity follows from (i), while the last one from (iii). \end{proof}

We now introduce the  ``generalized conjugation'' in $D_F {\mathcal U}$ following \S 4 of \cite{M}. For ${\bf x} = (x_1,...,x_k) \in {\mathbb C}^k$ define the automorphism $\Theta_F^{\bf x}$ of 
${D}_{F} {\mathcal U}$ in the following way.  For $u
\in {D}_{F}U$ set
$$
\Theta_F^{\bf x}(u):= \sum\limits_{
i_{1},\dots,i_{k} \geq 0} \binom{x_{1}}{i_{1}} \dots
\binom{x_{k}} {i_{k}} \,\mbox{ad}(f_{1})^{i_{1}}\dots
\mbox{ad}(f_{k})^{i_{k}}(u) \,f_{1}^{-i_{1}}\dots
f_{k}^{-i_{k}},
$$
 where $\binom{x}{i} :=
x(x-1)...(x-i+1)/i!$ for $x \in {\mathbb C}$ and $i \in {\mathbb Z}_{\geq 0}$. Note that for ${\bf x} \in {\mathbb Z}^k$,  we have
$\Theta_{F}^{\bf x}(u) = {\bf f}^{\mathbf x}u {\bf f}^{-\mathbf{x}}$, where ${\bf f}^{\bf x}:=f_{1}^{x_1}...f_{k}^{x_k} $. For a
${D}_{F} {\mathcal U}$-module $N$ by $\Phi^{\bf x}_{F} N$ we denote
the ${D}_{F} {\mathcal U}$-module $N$ twisted by $\Theta_{F}^{\bf x}$.  The action on  $\Phi^{\bf x}_{F} N$ is given by
 $$
u \cdot v^{\bf x} :=
 ( \Theta_{F}^{\bf x}(u)\cdot v)^{\bf x},
$$
where $u \in {D}_{F} {\mathcal U}$, $v \in N$, and $w^{\bf x}$
stands for the element $w$ considered as an element of
$\Phi^{\bf x}_{F} N$.
In the case ${\bf x} \in {\mathbb Z}^k$, there is a natural isomorphism of ${D}_{F} {\mathcal U}$-modules $M \to \Phi_{F}^{\bf x} M$ given by $m \mapsto ({\bf f}^{\bf x} \cdot m)^{\bf x}$ with inverse map defined by $n^{\bf x} \mapsto {\bf f}^{-\bf{x}} \cdot n$.
In view of this isomorphism, for ${\bf x} \in {\mathbb Z}^k$, we will identify $M$ with  $\Phi_{F}^{\bf x}M$, and for any ${\bf x} \in {\mathbb C}^k$ will write ${\bf f}^{\bf x} \cdot m$ (or simply ${\bf f}^{\bf x} m$) for $m^{-\bf{x}}$ whenever $m \in M$. 

The basic properties of the twisting functor $\Phi_{F}^{\bf x}$ on ${D}_{F} {\mathcal U}$-mod are summarized in the following lemma. The proofs are straightforward.

\begin{lemma} \label{lemma-conj} Let $F = \{ f_1,...,f_k\}$ be a set of locally ad-nilpotent commuting elements of ${\mathcal U}$, $M$ and $N$  be ${D}_{F} {\mathcal U}$-modules, $m \in M$, $u \in {\mathcal U}$, and  ${\bf x}, {\bf y} \in {\mathbb C}^k$.

{\rm(i)} $\Theta_{F}^{\bf x} \circ \Theta_{F}^{\bf y}  = \Theta_{F}^{\bf{x+y}}$, in particular
${\bf f}^{\bf x} \cdot ({\bf f}^{\bf y} \cdot m) =
{\bf f}^{\bf{x+y}} \cdot m$;

{\rm(ii)}  $\Phi^{\bf x}_{F} \Phi^{\bf y}_{F} =\Phi^{\bf{x+y}}_{F} $, in particular, $\Phi^{\bf x}_{F}
\Phi^{-\bf{x}}_{F} =\operatorname {Id}$ on the category of $D_F {\mathcal U}$-modules;

{\rm(iii)}   ${\bf f}^{\bf x} \cdot ( u \cdot ({\bf f}^{-{\bf x}} \cdot m) ) = \Theta_{F}^{\bf x}(u) \cdot m$;

{\rm(iv)}  $\Phi^{\bf x}_{F}$ is an exact functor;

{\rm(v)}  $M$ is simple  (respectively, injective) if and
only
 if $\Phi^{\bf x}_{F} M$ is simple (respectively, injective);  
 
{\rm(vi)}  ${\rm Hom}_{\mathcal U}(M,N) = {\rm Hom}_{\mathcal U}(\Phi^{\bf x}_{F} M, \Phi^{\bf x}_{F} N)$.

\end{lemma}

For any ${\mathcal U}$-module $M$, and ${\bf x} \in {\mathbb C}^k $ we define
the {\it twisted localization ${D}_{F}^{\bf x} M$ of $M$
relative to $F$ and $\bf x$} by ${D}_{F}^{\bf x} M:=
\Phi^{\bf x}_{F} {D}_{F} M $. The twisted localization is a exact functor from ${\mathcal U}$-mod to $D_F {\mathcal U}$-mod.

\subsection{Twisted localization of generalized weight $({\mathcal U}, {\mathcal H})$-modules}

In this subsection we apply the functor $D_F^{\bf x}$ to the category of generalized weight $({\mathcal U}, {\mathcal H})$-modules. 

\begin{lemma}
Assume that $f_i \in {\mathcal U}^{\bf a_i}$ for $\bf{a_i} \in Q$. 

(i) If $M$ is a generalized weight $({\mathcal U}, {\mathcal H})$-module, then $D_F M$ is a generalized weight $(D_F{\mathcal U}, {\mathcal H})$-module.

(ii) If $N$ is a generalized weight $(D_F {\mathcal U}, {\mathcal H})$-module then  ${\bf f}^{\bf{x}} m \in N^{(\lambda+\bf{xa})}$ whenever $m \in N^{(\lambda)}$, where 
${\bf xa} = x_1{\bf a_1}+...+ x_k{\bf a_k}$. In particular, $\Phi_F^{\bf x} N$ is a generalized weight $(D_F {\mathcal U}, {\mathcal H})$-module.
\end{lemma}

\subsection{Twisted localization in $({\mathcal D}, {\mathcal H})$-mod} \label{subsec-twloc-dh}

Consider now ${\mathcal U} = {\mathcal D} (n+1)$ and ${\mathcal H} = {\mathbb C}[t_0 \partial_0,...,t_n\partial_n]$. In this case we set $D_i^+ = D_{\{ \partial_i \}}$, $D_i^{-} = D_{\{ t_i \}}$, $D_i^{x,-}:=  D_{\{ t_i\} }^x$,   $D_i^{x,+} = D_{\{ \partial_i\} }^x$, and   $D_{i,j}^{x} = D_{\{ t_i \partial_j\} }^x$ for $x \in {\mathbb C}$ and $i \neq j$.

\begin{lemma}\label{relations} Assume that $i\in{\rm Int}(\nu)$ and
$J\in\mathcal P(\nu)$.

If $i\in J$, then
$$D^+_i ({\mathcal F}^{\log}_\nu(J))\simeq {\mathcal F}^{\log}_\nu(J),\,\,D^-_i ({\mathcal F}^{\log}_\nu(J))\simeq {\mathcal F}^{\rm
log}_\nu(J\setminus i).$$

If $i\notin J$, then
$$D^-_i ({\mathcal F}^{\log}_\nu(J))\simeq {\mathcal F}^{\rm
log}_\nu(J),\,\,D^+_i ({\mathcal F}^{\log}_\nu(J))\simeq {\mathcal
F}^{\log}_\nu(J\cup i).$$

\end{lemma}

\begin{proof} First, we have the relations
$$D^+_i\sigma_i=\sigma_i D^-_i,\,\,D^\pm_i\sigma_j=\sigma_jD^\pm_i\,\,\text{if}\,\,i\neq j.$$
Therefore it is sufficient to prove 
$$D^-_i ({\mathcal F}_\nu^{\log})\simeq {\mathcal F}_\nu^{\log},\,\,D^+_i ({\mathcal F}_\nu^{\log})\simeq {\mathcal F}_\nu^{\log}(\{i\}).$$
Note that $\mathcal D(n+1)=\mathcal D(n)\otimes D(1)$ and 
$${\mathcal F}_\nu^{\log}={\mathcal F}_{\nu,i}^{\log}\otimes\mathbb C[t_i,t_i^{-1},\log t_i]$$
where ${\mathcal F}_{\nu,i}^{\log}$ is the analogue of ${\mathcal F}_{\nu}^{\log}$ for the corresponding $\mathcal D(n)$.
Hence we need to prove the latter relation only for $\mathcal D(1)$. Note that
in this case we may assume $\nu=0$. Then 
${\mathcal F}^{\log}_\nu(\emptyset)=\mathbb C[t,t^{-1},\log t]$ and
it has an increasing socle filtration
$$\mathbb C[t,t^{-1},\log t]_0 = 0\subset\mathbb C[t,t^{-1},\log t]_1\subset\mathbb C[t,t^{-1},\log t]_2\subset\dots,$$
where
$$\mathbb C[t,t^{-1},\log t]_{2i}=\mathbb
C[t,t^{-1}]\otimes (\sum_{j=0}^{i-1}\mathbb C(\log t)^j), \; i\geq 1$$
and 
$$\mathbb C[t,t^{-1},\log t]_{2i+1}=\mathbb C[t,t^{-1},\log
t]_{2i}\oplus\mathbb C[t](\log t)^{i}, \; i \geq 0.$$
Note that $\mathbb C[t,t^{-1},\log t]_1$ is the maximal subspace
on which  $\partial$ acts locally nilpotently.
Hence
$$D_1^+(\mathbb C[t,t^{-1},\log t])=\mathbb C[t,t^{-1},\log
t]/\mathbb C[t,t^{-1},\log t]_1.$$
But $\mathbb C[t,t^{-1},\log t]/\mathbb C[t,t^{-1},\log t]_1$
has a unique simple submodule isomorphic to $\mathcal S_\nu(\{1\})$. Hence we
have a homomorphism $D_1^+({\mathcal F}^{\log}_\nu(\emptyset))\to {\mathcal F}^{\log}_\nu(\{1\})$.
It is injective since $D_1^+({\mathcal F}^{\log}_\nu(\emptyset))$ has a
unique simple submodule which maps to a non-zero space,  and it is surjective by
comparison of the socle filtration of both modules. Thus, we have
$D_1^+({\mathcal F}^{\log}_\nu(\emptyset))\simeq{\mathcal F}^{\log}_\nu(\{1\})$.

The isomorphism $D^-_i ({\mathcal F}_\nu^{\log})\simeq {\mathcal F}_\nu^{\log}$ follows  from the fact that $t_i$ is an invertible operator on 
${\mathcal F}_\nu^{\log}$.
\end{proof}

Recall that $\varepsilon_i \in {\mathbb C}^{n+1}$ are defined by $(\varepsilon_i)_j = \delta_{ij}$. Lemma \ref{relations} easily implies the following.
\begin{lemma} With notations as above
$$D^{x,+}_i ({\mathcal F}^{\log}_\nu(J))\simeq {\mathcal
F}^{\log}_{\nu- x \varepsilon_i}(J\cup i), \; D^{x,-}_i ({\mathcal F}^{\log}_\nu(J))\simeq {\mathcal
F}^{\log}_{\nu + x \varepsilon_i}(J\setminus i).$$
\end{lemma}

\begin{corollary} \label{cor-locid}
Let $J = \{ j_1,...,j_k\}$ be a nonempty proper subset of $\{0,...,n\}$, and let $\{i_1,...,i_{\ell} \} = \{0,...,n\} \setminus J$. Then for any $\mu, \nu \in {\mathbb C}^{n+1}$,
$$
D_{i_1}^{\nu_{i_1} - \mu_{i_1}, -}... D_{i_k}^{\nu_{i_\ell} - \mu_{i_\ell}, -} D_{j_1}^{\mu_{j_1} - \nu_{j_1}, +}...D_{j_k}^{\mu_{j_k} - \nu_{j_k}, +} \mathcal F^{\log}_{\mu} \simeq  \mathcal F^{\log}_{\nu} (J).
$$

\end{corollary}

\begin{lemma} \label{local-com} The following isomorphisms hold for every $M$ in $(\mathcal D, \mathcal H)$-mod, $i \neq j$, and $x,y \in {\mathbb C}$.

(i) $D^{x,\pm}_i D^{y,\pm}_j M \simeq D^{y,\pm}_j D^{x,\pm}_i M $.

(ii)$D^{x,\pm}_i D^{y,\pm}_i M \simeq D^{x+y,\pm}_i M$

(iii) $D^{x,-}_i D^{x,+}_j M \simeq D^{x,+}_j D^{x,-}_i M  \simeq D_{i,j}^x M$.

\end{lemma}
\begin{proof} The isomorphisms can be defined as follows\\
$t_i^{x} \cdot (t_j^y ) \cdot m) \mapsto t_j^y \cdot (t_i^x \cdot m)$ and  $\partial_i^{x} \cdot (\partial_j^y ) \cdot m) \mapsto \partial_j^y \cdot (\partial_i^x \cdot m)$  for (i);\\
$t_i^{x} \cdot (t_i^y \cdot m) \mapsto t_i^{x+y} \cdot m$ and $\partial_i^{x} \cdot (\partial_i^y \cdot m) \mapsto \partial_i^{x+y} \cdot m$ for (ii);\\
$ t_i^x  \cdot (\partial_j^{x}\cdot m) \mapsto (t_i \partial_j)^{x} \cdot m$ for (iii).
\end{proof} 

\begin{corollary} \label{cor-dij} Let  $\mu, \nu \in {\mathbb C}^{n+1}$ be such that $\mu_0 +...+ \mu_n = \nu_0+...+ \nu_n$. For any proper nonempty subset $J$ of $\{0,...,n\}$, there is a subset $S_J$ of $\{ (i,j) \; | \; i \notin J, j \in J\}$ with the following properties

(i) $S_J$ consists of $n$ elements;

(ii) For every $i \notin J$ (respectively, $j \in J$), there exists $j \in J$ (respectively, $i \notin J$) such that $(i,j) \in S_J$;

(iii)$$
\prod_{(i,j) \in S_J} D_{i,j}^{z(i,j)} \mathcal F^{\log}_{\mu} \simeq  \mathcal F^{\log}_{\nu} (J).
$$
for some $z(i,j) \in {\mathbb C}$ (note that the functors $ D_{i,j}^{z(i,j)}$ in the above product commute due to Lemma \ref{local-com}).
\end{corollary}
\begin{proof} Let $J = \{j_1,...,j_k \}$ and $\{ i_1,...,i_{\ell}\}$ is the complement of $J$ in $\{0,1,...,n \}$. Consider the system 
\begin{eqnarray*}
\sum_{s = 1}^{k} z(r,s) & = & \nu_{i_r} - \mu_{i_r}, \; r=1,...,\ell \\
\sum_{r = 1}^{\ell} z(r,s) & = & \mu_{j_s} - \nu_{j_s} , \; s = 1,...,k
\end{eqnarray*}
We can find a solution of the above system such that the set $S_J'$ of nonzero $z(i,j)$ satisfy conditions (i) and (ii) of the corollary.  To complete the proof we use the isomorphism in Corollary \ref{cor-locid} and the preceding lemma.
\end{proof}

\subsection{Twisted localization in $\mathcal{GB}$} \label{tw-loc-gb}

In this case we consider ${\mathcal U} = U(\mathfrak{sl} (n+1))$ and ${\mathcal H} = S(\mathfrak{h})$. The multiplicative sets will be always of the form $F = \langle e_{\alpha} \: | \: \alpha\in \Gamma \rangle$, where $\Gamma$ is a  set of $k$ commuting roots and $e_{\alpha}$ is in the $\alpha$-root space of $\mathfrak{sl} (n+1)$.  For ${\bf x} \in {\mathbb C}^{k}$, we write $D_{\Gamma}$ and $D_{\Gamma}^{\bf x}$ for $D_{F}$ and $D_{F}^{\bf x}$, respectively.

Recall that a subset $\Gamma$ of $\Delta$ is  a {\it set of commuting roots} if $\alpha, \beta \in \Gamma$ imply $[e_{\alpha}, e_{\beta}] = 0$. The maximal (with respect to inclusion) sets of commuting roots can be parametrized by the set  ${\mathcal P}_{-1}$ (cf. Section \ref{subsec-de}) of nonempty proper subsets of $\{0,1,...,n \}$. Indeed,
 for $J = \{ i_1,...,i_k \}$ in ${\mathcal P}_{-1}$,
$$
\Lambda_J : = \{ \varepsilon_i - \varepsilon_j \; | \; i \in J, j \notin J \}
$$
is a maximal set of commuting roots. Then one can check that the correspondence $J \mapsto \Lambda_J$ is a bijection between ${\mathcal P}_{-1}$ and the set of all maximal commuting sets of roots. In particular, if $\Gamma$ is a set of  $n$ linearly independent commuting roots then there is a unique $J \subset {\mathcal P}_{-1}$  such that $\Gamma \subset \Lambda_J$.

The following theorem is proven in \cite{M}.

\begin{proposition} \label{tloc-props}Every simple module in $\mathcal{GB}$ (equivalently, in $\mathcal{B}$) is isomorphic to $D_{\Gamma}^{\bf x} L$ for some simple highest weight module $L$ in $\mathcal{GB}$, a  set $\Gamma$ of $n$ commuting roots, and ${\bf x} \in {\mathbb C}^n$.

\end{proposition}

\section{Equivalence of categories of generalized weight ${\mathcal D}^E$-modules and  generalized bounded  $\mathfrak{sl}(n+1)$-modules}

\subsection{The functor $\Psi$} \label{subsec-psi}
Consider the homomorphism
$\psi:U(\mathfrak{sl}(n+1))\to \mathcal D (n+1)$ defined by $\psi(E_{ij})=t_i\partial_j$. The image of $\psi$ is contained in $\mathcal D^E$.
Using lift by $\psi$ any $\mathcal D^E$-module becomes 
$\mathfrak{sl}(n+1)$-module. Since $\psi(U(\mathfrak h))\subset\mathcal H$,  
one has a functor $\Psi:(\mathcal D^E, \mathcal H){\rm -mod}\to \overline{\mathcal{GB}}$.  
Obviously $\Psi$ is exact. 

In all statements below we assume that $\nu \in {\mathbb C}^{n+1}$.

\begin{lemma}\label{conn1} Let $|\nu|=a$. Then 
$$\Psi (^{\rm b}(\mathcal D^E, \mathcal H)_{\nu}^{a}{\rm -mod})\subset \mathcal{GB}^{\gamma (a\varepsilon_0)}_{\gamma(\nu)}.$$
\end{lemma}
\begin{proof} The inclusion 
$$\Psi (^{\rm b}(\mathcal D^E, \mathcal H)_{\nu}{\rm -mod})\subset \mathcal{GB}_{\gamma(\nu)}$$
follows immediately from the definition. Since
$\psi(Z(U(\mathfrak{sl}(n+1))))$ is contained in the center of $\mathcal D^E$ and the latter is generated by $E$,
$\Psi(M)$ admit the same central character for all $M\in ^{\rm b}(\mathcal D^E, \mathcal H)^{a}-{\rm mod}$. Furthermore, 
$\Psi(S_{a\varepsilon_0}(\emptyset))$ is a
highest weight module with highest weight $a\varepsilon_0$. Therefore 
$$\Psi (^{\rm b}(\mathcal D^E, \mathcal H)_{\nu}^{a}{\rm -mod})\subset \mathcal{GB}^{\gamma (a\varepsilon_0)}_{\gamma(\nu)}.$$
\end{proof}

In the case $\Gamma = \{ \alpha\}$ we will use the following notation: $D_{\alpha}^{x}  =  D_{\{e_{\alpha} \}}^{x}$ for $x \in {\mathbb C}$. 
\begin{lemma} \label{lemma-loc-psi}
If $M$ is a module in $(\mathcal D^E, \mathcal H){\rm -mod}$ and $x \in {\mathbb C}$, then $D_{\varepsilon_i - \varepsilon_j}^x \Psi (M) \simeq \Psi (D_{i,j}^x M)$.
\end{lemma}
\begin{proof}
The map $e_{\varepsilon_i - \varepsilon_j}^x \cdot m \mapsto (t_i \partial_j)^x \cdot m$ provides the desired isomorphism.
\end{proof}

\begin{remark} \label{rem-loc-inv}
{\it Using the above lemma and the fact that the functor  $\Psi:{}^{\rm b}(\mathcal D^E, \mathcal H){\rm -mod}\to \mathcal{GB}$ is surjective 
(which follows from Theorem \ref{connmain}), we can prove that for $\Gamma = \{ \alpha_1,...,\alpha_k\}$, 
the localization functor $D_{\Gamma}^{\bf x}$ on $\mathcal{GB}$ depends only on ${\rm Span} \Gamma$ and the weight $\sum_{i=1}^k x_i\alpha_i$. 
More precisely, if  $\Gamma' = \{ \beta_1,...,\beta_{\ell} \}$ is another set of commuting roots of $\mathfrak{sl} (n+1)$ such that 
${\rm Span}_{\mathbb C} \Gamma =  {\rm Span}_{\mathbb C} \Gamma'$ and if ${\bf y} \in {\mathbb C}^{\ell}$ be such that 
$\sum_{i=1}^k x_i \alpha_i =\sum_{i=1}^{\ell} y_i \beta_i$, then $D_{\Gamma}^{\bf x}M \simeq D_{\Gamma'}^{\bf y}M$ for every $M$ in ${\mathcal GB}$. 

Here is a short proof of this statement.  Denote $I= \{ i \; | \; \varepsilon_i - \varepsilon_j \in \Gamma, \mbox{ for some } j\}$, and $J = \{ j \; | \; \varepsilon_i - \varepsilon_j \in \Gamma, \mbox{ for some } i\}$ and define $I'$ and $J'$ in a similar way from $\Gamma'$. Then ${\rm Span}_{\mathbb C} \Gamma =  {\rm Span}_{\mathbb C} \Gamma'$ implies that $I = I'$ and $J = J'$. Let now $M=\Psi (N)$. Applying multiple times Lemma \ref{local-com} and Lemma \ref{lemma-loc-psi} we obtain that 
$$
D_{\Gamma}^{\bf x}M \simeq D_{\Gamma}^{\bf x} \Psi (N) \simeq  \Psi \left( \prod_{i \in I} D_i^{z_i,+} \prod_{j\in J}D_{j}^{z_j,-}(N)\right ) \simeq  D_{\Gamma'}^{\bf y} \Psi (N) \simeq D_{\Gamma'}^{\bf y}M, 
$$
where $\sum_k x_k \alpha_k = \sum_{i \in I} z_i \varepsilon_i - \sum_{j \in J} z_j \varepsilon_j$.}

\end{remark}

\subsection{The generic case}\label{subsec-gen-case}
\begin{lemma}\label{conn2}  Suppose  $a \notin \mathbb Z$, or
$a=-1,...,-n$ for $n\geq 2$ ($a$ is arbitrary for $n=1$). 
Any simple 
$S\in \mathcal{GB}^{\gamma (a\varepsilon_0)}$ is isomorphic to $\Psi(S')$
for some simple $S'\in {}^{\rm b}(\mathcal D^E, \mathcal H)_{\nu}^{a}{\rm -mod}$.
\end{lemma}
\begin{proof} By Proposition \ref{tloc-props}, every simple $S$ in $\mathcal{GB}^{\gamma (a\varepsilon_0)}$ is isomorphic to $D_\Gamma^{\bf x}L$, where $L$ is a simple highest weight module in $ \mathcal{GB}^{\gamma (a\varepsilon_0)}$, $\Gamma$ is a commuting set of roots and ${\bf x} \in {\mathbb C}^n$. 
As mentioned in the proof of Lemma \ref{conn1},  $L=\Psi(S_{a\varepsilon_0}(\emptyset))$. Therefore the annihilators in 
$U(\mathfrak{sl}(n+1))$ of all simple objects
 $S$ in $\mathcal{GB}^{\gamma (a\varepsilon_0)}$ coincide. So, every such $S$ is annihilated by the kernel of $\psi$,  and thus $S=\Psi(S')$ for some  
$S'\in {}^{\rm b}(\mathcal D^E, \mathcal H)_{\nu}^{a}{\rm -mod}$.
\end{proof}
We use the notation $\mathbb C^{k}[u_0,\dots,u_n]$ and  
$\mathbb C^{<k}[u_0,\dots,u_n]$ for the
polynomials of degree $k$,  and of degree strictly less than $k$, respectively.

\begin{lemma}\label{lemma-ind-log}  Let ${\rm Int}(\nu)=\emptyset$,  
$n\geq 2$ or $a \neq -1$. 

(a) $\Psi(\Gamma_a(\mathcal F_{\nu}))$ is the unique simple submodule 
of $\Psi(\Gamma_a(\mathcal F^{\log}_{\nu}))$. 

(b) For any positive $k$ 
$${\rm soc}^k \Psi(\Gamma_a(\mathcal
F^{\log}_{\nu})) \simeq \Psi(\Gamma_a(\mathcal F^{\log, <k}_{\nu})),$$
where 
$\mathcal F^{\log, <k}_{\nu}=\mathcal F_{\nu}\otimes\mathbb C^{<k}[u_0,\dots,u_n]$.
\end{lemma}
\begin{proof}
(a) Assume that $\Psi(\Gamma_a(\mathcal F^{\log}_{\nu}))$ has a
simple submodule $M$ distinct from  $\Psi(\Gamma_a({\mathcal
F}_\nu))$. Every simple module in
$\mathcal{GB}^{\gamma(a\varepsilon_0)}_{\gamma(\nu)}$ is a weight
module. 
Let us fix a weight vector $m = f(u_0,...,u_n)t_0^{\lambda_0}...t_n^{\lambda_n}$
in $M$, where $f(u_0,...,u_n) \in {\mathbb C}[u_0,...,u_n]$ and $u_i = \log
t_i$. 
Since $t_i \partial_i - t_j \partial_j$ act as a scalar multiple on 
$m$ we easily conclude that $f(u_0,\dots,u_n) = g(u_0+...+u_n)$ for some
polynomial $g$. Thus, we have that 
$$S_{\mathfrak{h}}(\Psi(\Gamma_a(\mathcal F^{\log}_{\nu})))=
\Psi(\Gamma_a(\mathcal F_{\nu}\otimes\mathbb C[u])),$$
where $u = u_0+...+u_n$ and $S_{\mathfrak{h}}$ stands for the functor $\mathcal{GB}^{\gamma (a \varepsilon_0)} \to  \mathcal{B}^{\gamma (a \varepsilon_0)}$ mapping a module to its submodule consisting of all $\mathfrak{h}$-eigenvectors. Now, the statement follows from Lemma 5.4 of \cite{GS2}. 

(b) We apply induction on $k$. The base case $k=1$ follows from part
(a). Assume that the statement holds for all $k < \ell$. 
In the case $k= \ell$ consider the quotient $N= \Psi(\Gamma_a(\mathcal
F^{\log}_{\nu}))/ \mbox{soc}^{\ell} \Psi(\Gamma_a(\mathcal
F^{\log}_{\nu}))$. By the same argument as in (a) 
$$S_{\mathfrak{h}}(N)=\Psi(\Gamma_a(\mathcal F_{\nu}\otimes\mathbb
C^{\ell}[u_0,\dots, u_n]\otimes\mathbb C[u])).$$
Therefore $S_{\mathfrak{h}}(N)$ is isomorphic to the direct sum of $\binom{n+k}{k}$
copies of $\Psi(\Gamma_a(\mathcal F_{\nu}\otimes\mathbb C[u]))$ and
the socle of  $S_{\mathfrak{h}}(N)$ coincides with 
$\Psi(\Gamma_a(\mathcal F_{\nu}\otimes\mathbb C^{\ell}[u_0,\dots, u_n]))$
by (a).
\end{proof}

\begin{remark}\label{specialcase} Let us explain why Lemma \ref{lemma-ind-log} is false in the case $n=1,a=-1$.
Indeed, let $\nu=c\varepsilon_0+(-1-c)\varepsilon_1$ for some $c\in\mathbb C\setminus\mathbb Z$. Let $\varphi:c+\mathbb Z\to\mathbb C$ 
be a function satisfying the condition $\varphi(x+1)-\varphi(x)=\frac{-1}{x+1}$ for all $x\in c+\mathbb Z$. Then it is easy to check
that the elements $t_0^xt_1^{-1-x}\log{t_0t_1}+\varphi(x)t_0^xt_1^{-1-x}$ for  $x\in c+\mathbb Z$ span a simple submodule of $\Psi(\Gamma_a(\mathcal F^{\log}_{\nu}))$.
On the other hand, $\Psi(\Gamma_a(\mathcal F_{\nu}))$ is another simple submodule in $\Psi(\Gamma_a(\mathcal F^{\log}_{\nu}))$. Hence 
$\Psi(\Gamma_a(\mathcal F^{\log}_{\nu}))$ does not have a simple socle.
\end{remark}

\begin{theorem}\label{connmain}  Assume that $a \notin \mathbb Z$ or
$a=-1,...,-n$ for $n\geq 2$ and  $a\neq -1$ for $n=1$. 
Then $\Psi$ provides equivalence between 
$^{\rm b}(\mathcal D^E, \mathcal H)^{a}-{\rm mod}$ 
and $\mathcal{GB}^{\gamma (a\varepsilon_0)}$ and between
$(\mathcal D^E, \mathcal H)^{a}-{\rm mod}$ and  
$\overline{\mathcal{GB}}^{\gamma (a\varepsilon_0)}$.

Moreover, $\Psi$ provides equivalence between $_{\rm s}^{\rm b}(\mathcal D^E, \mathcal H)^{a}-{\rm mod}$ and  
$\mathcal{B}^{\gamma (a\varepsilon_0)}-{\rm mod}$, as well as,  between $_{\rm s}(\mathcal D^E, \mathcal H)^{a}-{\rm mod}$ and  
$\overline{\mathcal{B}}^{\gamma (a\varepsilon_0)}-{\rm mod}$.
\end{theorem}
\begin{proof} 
We prove the theorem in several steps. First, we  prove it  for the cuspidal blocks.
\begin{lemma}\label{conncusp} Assume that $a \notin \mathbb Z$ or
$a=-1,...,-n$ for $n\geq 2$ and  $a\neq -1$ for $n=1$.  Let ${\rm Int}(\nu)=\emptyset$.

(a) $\Psi(\Gamma_a(\mathcal F^{\log}_{\nu}))$ is an indecomposable
injective in $\overline {\mathcal{GB}}$.

(b) $\Psi$ provides equivalence between 
$^{\rm b}(\mathcal D^E, \mathcal H)^{a}_{\nu}-{\rm mod}$ 
and $\mathcal{GB}^{\gamma (a\varepsilon_0)}_{\gamma(\nu)}$, and between
$(\mathcal D^E, \mathcal H)^{a}_{\nu}-{\rm mod}$ 
and $\overline{\mathcal{GB}}^{\gamma (a\varepsilon_0)}_{\gamma(\nu)}$.
\end{lemma}
\begin{proof} 
(a) Both categories $(\mathcal D^E, \mathcal H)^{a}_{\nu}-{\rm mod}$ 
and $\overline{\mathcal{GB}}^{\gamma (a\varepsilon_0)}_{\gamma(\nu)}$ 
have one up to isomorphism simple module
$S=\Gamma_a(\mathcal F_\nu)$ and $\Psi(S)$, respectively. 
As follows from \cite{MS},  
$\overline{\mathcal{GB}}^{\gamma (a\varepsilon_0)}_{\gamma(\nu)}$
is equivalent to the category of locally nilpotent 
$\mathbb C[z_0,\dots,z_n]$-modules. Let $J$ be the indecomposable
injective in $\overline{\mathcal{GB}}^{\gamma(a\varepsilon_0)}_{\gamma(\nu)}$.
The length of $\mbox{soc}^k(J)$ is the same as the length of 
 $\mbox{soc}^k (\Psi(\Gamma_a(\mathcal F^{\log}_{\nu})))$. On the
 other hand, we have an injective homomorphism
$j:\Psi(\Gamma_a(\mathcal F^{\log}_{\nu}))\to J$ induced by the
embedding $\Psi(S)\to J$. Since the terms of the socle filtration of
$J$ and $\Psi(\Gamma_a(\mathcal F^{\log}_{\nu}))$ have the same
length, $j$ is an isomorphism.

Part (b) follows  from (a).
\end{proof}

\begin{lemma}\label{conninj} Assume that $\nu\in\mathbb C^{n+1}$ and let 
$J\in\mathcal P(\nu)$ be such that $J\neq\emptyset,\{0,\dots,n\}$ whenever 
$a \in \mathbb Z$ and $\nu\in\mathbb Z^{n+1}$.
Then the functor $\operatorname{Hom}_{\mathfrak g}(\cdot,\Psi(\Gamma_a(\mathcal F^{\log}_\nu(J))))$
is exact on modules in the category $\mathcal{GB}^{\gamma(a\varepsilon_0)}$.
\end{lemma}
\begin{proof} Let us first fix $\mu$ with  $\rm{Int}(\mu) = \emptyset$ and such that $\mu_0 + ... + \mu_n = \nu_0 + ...+ \nu_n$. We use  Corollary \ref{cor-dij} and find a subset $S_J$ of $\{ (i,j) \; | \; i \notin J, j \in J\}$ that satisfy (i)-(iii) of that corollary. In particular, 
$$
\mathcal F^{\log}_{\nu} (J) \simeq \prod_{(i,j) \in S_J} D_{i,j}^{z(i,j)} \mathcal F^{\log}_{\mu},
$$
where 
$z(i,j)$ are such that  $\sum_{s = 1}^{\ell} z(r,s) = \nu_{i_r} - \mu_{i_r}, \; \sum_{r = 1}^{k} z(r,s) =  \mu_{j_s} - \nu_{j_s}$. The conditions (i) and (ii) of orollary \ref{cor-dij} imply that $\Sigma_J  := \{ \varepsilon_i - \varepsilon_j \; | \; (i,j) \in S_J\}$ is a  set of $n$ linearly independent commuting roots of $\mathfrak{sl} (n+1)$. Also, as one can easily show, the functor $\Gamma_a$ commutes with the twisted localization functors $D_{i,j}^x$.  Hence
$$
\Psi ( \Gamma_ a (\prod_{(i,j) \in S_J} D_{i,j}^{z(i,j)} \mathcal F^{\log}_{\mu})) \simeq  \Psi (\prod_{(i,j) \in S_J} D_{i,j}^{z(i,j)}( \Gamma_a ( {\mathcal  F}^{\log}_{\mu})) ) \simeq {D}_{\Sigma_J}^{\bf x}  ( \Psi ( \Gamma_a ( {\mathcal  F}^{\log}_{\mu}) )),
$$
where $\nu - \mu = \sum_{i=1}^n x_i \alpha_i$ and $\Sigma_J = \{ \alpha_1,...,\alpha_n\}$.
The last isomorphism follows from the identities for $z(r,s)$ and Lemma \ref{lemma-loc-psi}.
Therefore
$$\Psi(\Gamma_a(\mathcal F^{\log}_\nu(J)))={D}_{\Sigma_J}^{\bf x}(\Psi(\Gamma_a(\mathcal F^{\log}_{\mu}))).$$

Then using Lemma \ref{lem-loc-map} (iv) and Lemma \ref{lemma-conj} (vi) we have 
$$\operatorname{Hom}_{\mathfrak g}(X,{D}_{\Sigma_J}^{\bf x}\Psi(\Gamma_a(\mathcal F^{\log}_{\mu})))\simeq
\operatorname{Hom}_{\mathfrak g}({D}_{\Sigma_J} (X),{D}_{\Sigma_J}^{\bf x}\Psi(\Gamma_a(\mathcal F^{\log}_{\mu})))\simeq$$
$$\operatorname{Hom}_{\mathfrak g}({D}_{\Sigma_J}^{\bf x}(X),\Psi(\Gamma_a(\mathcal F^{\log}_{\mu}))).$$
Since ${D}_{\Sigma_J}^{\bf x}(\cdot)$ is exact, the statement follows from
Lemma \ref{conncusp}(a).
\end{proof}

Now we are ready to complete the proof of Theorem \ref{connmain}. 
Note first that any simple $S\in \mathcal{GB}^{\gamma(a\varepsilon_0)}$
is isomorphic to $\Psi(S')$ for some simple 
$S'\in {}^{\rm b}(\mathcal D^E, \mathcal H)^{a}-{\rm mod}$ 
by Lemma \ref{conn2}.

Next, by Lemma \ref{conninj}, $\Psi$ maps indecomposable injectives to
indecomposable injectives. Hence Theorem \ref{connmain} follows from Theorem
\ref{antieq}. The statement for weight modules can be easily obtained by
applying the functors $S_{\mathcal H'}$ and $S_{\mathfrak{h}}$.
\end{proof}
\subsection{The singular block for $\mathfrak{sl}(2)$-modules}
Now we  consider the case $n=1$ and $a=-1$.
\begin{proposition} Let $n=1$. The category 
$\mathcal{GB}^{\gamma(-\varepsilon_0)}_{\gamma(-\varepsilon_0)}$ is equivalent to the category
of nilpotent representation of the quiver
$$
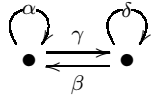
\xymatrix{\bullet  \ar@(ul,ur)|{\alpha} \ar@<0.5ex>[r]^{\gamma}& \bullet
\ar@<0.5ex>[l]^{\beta}  \ar@(ul,ur)[]|{\delta}}
$$
with relations $\alpha\beta\gamma=\beta\gamma\alpha$ and
$\delta\gamma\beta=\gamma\beta\delta$,
and the category $\mathcal{B}^{\gamma(-\varepsilon_0)}_{0}$ is 
equivalent to the category
of nilpotent representation of the quiver
$$
\xymatrix{\bullet \ar@<0.5ex>[r]& \bullet
\ar@<0.5ex>[l] }
$$
Moreover, the former category is wild, while the latter is tame.
\end{proposition}
\begin{proof} First, consider the cuspidal block 
$\overline{\mathcal{GB}}^{\gamma(-\varepsilon_0)}_{\gamma(\nu)}$. This block has only
one (up to isomorphism) simple object and one indecomposable injective $J$. As follows from \cite{MS},
${\rm End}_{\mathfrak{sl}(2)}(J)=\mathbb C[[z_0,z_1]]$. 
On the other hand, $\overline{\mathcal{GB}}^{\gamma(-\varepsilon_0)}_{\gamma(-\varepsilon_0)}$
has two simple modules $S_+$ and $S_-$ and therefore two
indecomposable injectives  $J_+$ and $J_-$. Note that $S_\pm$ are
Verma modules with respect to the two opposite Borel subalgebras of 
$\mathfrak{sl}(2)$. It is an easy exercise to check that
\begin{equation}\label{exteq}
{\rm Ext}^1_{\mathcal{GB}}(S_+,S_+)={\rm Ext}^1_{\mathcal{GB}}(S_+,S_-)=
{\rm Ext}^1_{\mathcal{GB}}(S_-,S_-)={\rm Ext}^1_{\mathcal{GB}}(S_-,S_+)=
\mathbb C.
\end{equation}
This already implies that the quiver of 
$\mathcal{GB}^{\gamma(-\varepsilon_0)}_{\gamma(-\varepsilon_0)}$ is as above.

To find the relations, we use the fact that
$J_{\pm}=D_{\pm \alpha}^{x} J$ where $\gamma(\varepsilon_0)-\nu=\pm x\alpha $(cf. Lemma \ref{lemma-conj} (v)). Thus, by Lemma \ref{lemma-conj} (vi)
$${ \rm End}_{\mathfrak{sl}(2)}(J_\pm)\simeq
{\rm End}_{\mathfrak{sl}(2)}(J)\simeq\mathbb C[[z_0,z_1]].$$
By (\ref{exteq}) we have 
$\beta\gamma\in {\rm End}_{\mathfrak{sl}(2)}(J_+)=\mathbb C[[z_0,z_1]].$
Therefore after a suitable change of variables, without loss of
generality, we may assume $\beta\gamma=z_0$. Let  $\alpha=z_1$. Then we
have the relation  $\alpha\beta\gamma=\beta\gamma\alpha$.  We define $\delta$ and obtain the second relation in the same way. Clearly, there are no other relations.

To see that the category  $\mathcal{GB}^{\gamma(-\varepsilon_0)}_{\gamma(-\varepsilon_0)}$
is wild set $\beta=0$. Then, the corresponding quotient algebra is the
algebra of the following wild quiver without relations.

$$
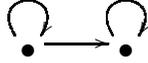
\xymatrix{\bullet  \ar@(ul,ur)|{} \ar@<0.5ex>[r]^{}& \bullet
  \ar@(ul,ur)[]|{}}
$$

To prove the statement for $\mathcal{B}^{\gamma(-\varepsilon_0)}_{\gamma(-\varepsilon_0)}$
we first observe that it has the same simple objects as  
$\mathcal{GB}^{\gamma(-\varepsilon_0)}_{\gamma(-\varepsilon_0)}$ and
\begin{equation}\label{exteq1}
{\rm Ext}^1_{\mathcal{B}}(S_+,S_+)={\rm Ext}^1_{\mathcal{B}}(S_-,S_-)=0,\,\,
{\rm Ext}^1_{\mathcal{B}}(S_+,S_-)={\rm Ext}^1_{\mathcal{B}}(S_-,S_+)=
\mathbb C.
\end{equation}
If we denote by $J'_\pm$ the injective hulls of $S_\pm$ in
$\overline{\mathcal{B}}^{\gamma(-\varepsilon_0)}_{\gamma(-\varepsilon_0)}$, then 
${\rm End}_{\mathfrak{sl}(2)}(J'_\pm)\simeq\mathbb C[z]$ by \cite{GS2}
and  Lemma \ref{lemma-conj} (vi) as above.
In this way we obtain the second quiver with no relations, 
that is a well-known tame quiver.
\end{proof}

\bigskip
{\bf Appendix: Some useful commutative diagrams} 

\bigskip

Diagram 1. 
$$
\xymatrix{(\mathcal D, \mathcal H){\rm -mod}  \ar@<0.5ex>[r]^{\Gamma_a} \ar[d]^{S_{\mathcal H'}}& (\mathcal D^E, \mathcal H)^{a}{\rm -mod}   \ar@<0.5ex>[l]^{\Phi, \Phi'} \ar[d]^{S_{\mathcal H'}} \ar[r]^{\Psi}  &  \overline{\mathcal{GB}}^{\gamma(a \varepsilon_0)}{\rm -mod}   \ar[d] ^{S_{\mathfrak h}}\\
{}_{\rm s}(\mathcal D, \mathcal H){\rm -mod}    \ar@<0.5ex>[r]^{\Gamma_a}  & {}_{\rm s}(\mathcal D^E, \mathcal H)^{a}{\rm -mod}    \ar@<0.5ex>[l]^{\Phi, \Phi'}  \ar[r]^{\Psi}  &  \overline{\mathcal{B}}^{\gamma(a \varepsilon_0)}{\rm -mod}  }
$$

\medskip
Diagram 2.

$$
\xymatrix{  ^{\rm b}(\mathcal D^E, \mathcal H)^{a}{\rm -mod} \ar[d]^{S_{\mathcal H'}} \ar[r]^{\Psi}  &  \mathcal{GB}^{\gamma(a \varepsilon_0)}{\rm -mod}   \ar[d] ^{S_{\mathfrak h}}\\
{}_{\rm s}^{\rm b}(\mathcal D^E, \mathcal H)^{a}{\rm -mod}    \ar[r]^{\Psi}  &  \mathcal{B}^{\gamma(a \varepsilon_0)}{\rm -mod}  }
$$

Diagram 3.
$$
\xymatrix{(\mathcal D, \mathcal H)_{\nu}{\rm -mod}  \ar@<0.5ex>[r]^{\Gamma_a} \ar[d]^{D_{i,j}^x}& (\mathcal D^E, \mathcal H)_{\nu}^{a}{\rm -mod}   \ar@<0.5ex>[l]^{\Phi, \Phi'} \ar[d]^{D_{i,j}^x} \ar[r]^{\Psi}  &  \overline{\mathcal{GB}}_{\gamma(\nu)}^{\gamma(a \varepsilon_0)}{\rm -mod}   \ar[d] ^{D_{\varepsilon_i - \varepsilon_j}^x}\\
(\mathcal D, \mathcal H)_{\mu}{\rm -mod}   \ar@<0.5ex>[r]^{\Gamma_a}  & (\mathcal D^E, \mathcal H)^a_{\mu}{\rm -mod}   \ar@<0.5ex>[l]^{\Phi, \Phi'}  \ar[r]^{\Psi}  &  \overline{\mathcal{GB}}_{\gamma(\mu)}^{\gamma(a \varepsilon_0)}{\rm -mod}  }
$$

where $\mu = \nu+x(\varepsilon_i - \varepsilon_j)$.

\bibliography{ref,outref,mathsci}
\def\cprime{$'$} \def\cprime{$'$} \def\cprime{$'$} \def\cprime{$'$}
  \def\cprime{$'$}
\providecommand{\bysame}{\leavevmode\hbox
to3em{\hrulefill}\thinspace}
\providecommand{\MR}{\relax\ifhmode\unskip\space\fi MR }
\providecommand{\MRhref}[2]{%
  \href{http://www.ams.org/mathscinet-getitem?mr=#1}{#2}
} \providecommand{\href}[2]{#2}

\end{document}